\documentclass[11pt]{amsart}
\usepackage[colorlinks,citecolor=blue,urlcolor=black,linkcolor=black]{hyperref}
\usepackage{amsfonts,amssymb,amsmath,amsthm,bbm,enumitem}

\newtheorem{mainthm}{Theorem}

\newtheorem{thm}{Theorem}[section]
\newtheorem{lemma}[thm]{Lemma}
\newtheorem{cor}[thm]{Corollary}
\newtheorem{claim}{Claim}[thm]
\newtheorem{prop}[thm]{Proposition}
\newtheorem{fact}[thm]{Fact}
\newtheorem{subclaim}{Subclaim}[claim]

\theoremstyle{definition}
\newtheorem{defn}[thm]{Definition}
\newtheorem{notation}[thm]{Notation}
\newtheorem{q}[thm]{Question}

\theoremstyle{remark}
\newtheorem{remark}[thm]{Remark}

\renewcommand{\mid}{\mathrel{|}\allowbreak}
\renewcommand{\restriction}{\mathbin\upharpoonright}

\DeclareMathOperator{\acc}{acc}
\DeclareMathOperator{\reg}{Reg}
\DeclareMathOperator{\nacc}{nacc}
\DeclareMathOperator{\cf}{cf}
\DeclareMathOperator{\dom}{dom}
\DeclareMathOperator{\im}{Im}
\DeclareMathOperator{\otp}{otp}
\DeclareMathOperator{\h}{ht}
\DeclareMathOperator{\p}{P}

\newcommand\s{\subseteq}
\newcommand\T{\mathcal T}
\newcommand\sq{\sqsubseteq}
\newcommand\br{\blacktriangleright}
\newcommand*\axiomfont[1]{\textsf{\textup{#1}}}
\newcommand\one{\mathbbm{1}}
\newcommand\forces{\Vdash}
\newcommand\zfc{\axiomfont{ZFC}}
\newcommand\pmea{\axiomfont{PMEA}}
\newcommand\ma{\axiomfont{MA}_{\omega_1}}
\newcommand\ch{\axiomfont{CH}}
\newcommand\ns{\textup{NS}}
\newcommand\VisL{\ensuremath{\axiomfont{V}=\axiomfont{L}}}
\newcommand\LE{\trianglelefteq}
\renewcommand\ll{<_{\mathbb{Q}_\kappa}}
\newcommand{\AlephOrBeth}{\ifnum\value{enumi}=1\aleph\else\beth\fi}

\makeatletter
\newcommand{\leqdot}{\mathrel{\mathpalette\leqdot@{\leq{r}}}\leq}
\newcommand{\leqdot@}[2]{\leqdot@@#1#2}
\newcommand{\leqdot@@}[3]{\begingroup\sbox\z@{$\m@th#1#2$}\makebox[0pt][l]{\makebox[\wd\z@][#3]{\raisebox{0.15\ht\z@}[0pt][0pt]{$\m@th#1\cdot$}}}\endgroup}
\makeatother

\title{The power of trees}

\author{Ari Meir Brodsky}
\address{Mathematics Department, Shamoon College of Engineering, 56 Bialik St., Be'er Sheva, Israel.}
\urladdr{\url{https://en.sce.ac.il/faculty/ari_brodsky}}
\email{arimebr@sce.ac.il}

\author{Assaf Rinot}
\address{Department of Mathematics, Bar-Ilan University, Ramat-Gan 52900, Israel.}
\urladdr{http://www.assafrinot.com}

\author{Shira Yadai}
\address{Department of Mathematics, Bar-Ilan University, Ramat-Gan 52900, Israel.}
\urladdr{https://orcid.org/0009-0003-9757-3349}
\email{greenss@biu.ac.il}


\begin{document}
\begin{abstract} We give two consistent constructions of trees $T$ whose finite power $T^{n+1}$ is sharply different from $T^n$:
\begin{itemize}
\item An $\aleph_1$-tree $T$ whose interval topology $X_T$ is perfectly normal, but $(X_T)^2$ is not even countably metacompact.
\item For an inaccessible $\kappa$ and a positive integer $n$, a $\kappa$-tree such that all of its $n$-derived trees are Souslin and all of its $(n+1)$-derived trees are special.
\end{itemize}
\end{abstract}
\date{Preprint as of January 6, 2026. For updates, visit \textsf{http://p.assafrinot.com/68}.}
\maketitle

\section{Introduction}\label{section:introduction}
This paper is a contribution to the study of features of structures that are not preserved by taking products.
As a simple example, the Sorgenfrey line \cite{MR20770} constitutes a normal topological space whose square is not normal.
As a more substantial example, a space is \emph{Dowker} iff it is normal, yet its product with the unit interval is not normal.
The naming comes from Dowker's theorem \cite{dowker}
that $X\times[0,1]$ is normal iff $X$ is normal and countably metacompact.
Recall that a topological space is \emph{countably metacompact (cmc)} iff
every countable open cover admits a point-finite open refinement.
A property stronger than cmc is that of being \emph{perfect}: a topological space is perfect iff all of its closed subsets are $G_\delta$.
An even stronger property, \emph{perfectly normal}, is equivalent to the conjunction of perfect and normal.

Of special interest is whether a tree $\mathbf T=(T,{<_T})$ equipped with
the \emph{interval topology},\footnote{The definition of the interval topology may be found in Section~\ref{secthma}.}
denoted $X_{\mathbf T}$, can satisfy the above properties.
To compare, while there are $\zfc$ examples of Dowker spaces \cite{Rudin_first_Dowker_ZFC_example,Balogh_space,MR1605988,MR4900695},
a space of the form $X_{\mathbf T}$ can never be Dowker \cite{MR1607401}.

Any antichain of $\mathbf T$ is closed discrete in $X_{\mathbf T}$, so if $\mathbf T$ is a special $\aleph_1$-tree, then $X_{\mathbf T}$ is perfect.
Nyikos (see \cite[Theorem~4.1]{MR577767}) proved that almost-Souslin $\aleph_1$-trees are cmc,
and Hanazawa \cite[Theorem~3]{MR679075} proved that $\mathbb R$-embeddable almost-Souslin $\aleph_1$-trees are perfect.\footnote{Every perfect $\aleph_1$-tree is $\mathbb R$-embeddable and cmc.
A Souslin tree cannot be $\mathbb R$-embeddable, hence the need to focus on finer notions.}
It follows that a non-cmc $\aleph_1$-tree must be quite unusual, as it can be neither special nor almost-Souslin.
A construction of such a spectacular creature was given first by Fleissner assuming $\diamondsuit^+(\omega_1)$ \cite[\S3]{MR577767},
and then by Hanazawa assuming $\diamondsuit(\omega_1)$ \cite[Theorem~1]{MR679075}.

The first main result of this paper is an extension of the above works to control not only $X_{\mathbf T}$, but also $(X_{\mathbf T})^2$, and in an independent way.
There are quite a few Boolean combinations of what $X_{\mathbf T}$ and $(X_{\mathbf T})^2$ could satisfy,
so we decided to focus on one particular combination that we feel demonstrates well how the construction machinery developed in \cite{paper65} can contribute to settling this kind of classical questions.
Specifically, we establish (consistently) the non-productivity of cmc in the class of spaces determined by the interval topology on trees, as follows.

\begin{mainthm}\label{thma}
Suppose that $\diamondsuit^*(\omega_1)$ holds. Then there exists an $\mathbb R$-embeddable almost-Souslin $\aleph_1$-tree $\mathbf T$ such that:
\begin{itemize}
\item $X_{\mathbf T}$ is perfectly normal;
\item $(X_{\mathbf T})^2$ is not cmc.
\end{itemize}
\end{mainthm}

We now turn to describe the second main result of this paper, dealing with the nonproductivity of the property of being nonspecial in the class of trees.\footnote{See Definition~\ref{defn41} and Remark~\ref{rmk42}.}
The literature has several consistent examples of $\lambda^+$-Souslin trees whose squares are special (see \cite{MR419229} or \cite[Proposition~1.4.15]{MR2638327} for $\lambda:=\aleph_0$,
\cite[\S4]{sh:221} for $\lambda$ singular, and \cite[\S6]{paper65} for a uniform treatment for all $\lambda$).
In this paper, we give the first consistent example of a Souslin tree of inaccessible height whose square is special.
This is obtained by introducing a new instance $\p_{<}(\ldots)$ of the proxy principle from \cite{paper22}, proving that this instance is consistent,
and presenting a construction of the desired tree from it, as follows.

\begin{mainthm}\label{thmb} Suppose that $\kappa$ is a regular uncountable cardinal.
\begin{enumerate}[label=\textup{(\arabic*)}]
\item There is a ${<}\kappa$-strategically-closed forcing $\mathbb P$ of size $\kappa^{<\kappa}$ such that the proxy principle $\p_{<}(\kappa,2,\allowbreak{\sq},\kappa)$ holds in the generic extension by $\mathbb P$.
\item If $\p_{<}(\kappa,2,\allowbreak{\sq},\kappa)$ holds,
then for every positive integer $n$, there exists a $\kappa$-Souslin tree $\mathbf T$ such that:
\begin{itemize}
\item all $n$-derived trees of $\mathbf T$ are Souslin;
\item all $(n+1)$-derived trees of $\mathbf T$ are special.
\end{itemize}
\end{enumerate}
\end{mainthm}

To those familiar with $C$-sequences and proxy principles, we explain here the obstacle we had to overcome in order to obtain such a tree for an inaccessible cardinal $\kappa$.
As listed in~\cite[Fact~1]{paper29}, $\kappa$-trees are intimately connected with $C$-sequences over $\kappa$.
In particular, it is a classical theorem going back to Jensen (see~\cite[p.~283]{MR309729}) that there is a special $\lambda^+$-Aronszajn tree
iff there is a $C$-sequence $\vec C=\langle C_\alpha\mid\alpha<\lambda^+\rangle$ satisfying the following:
\begin{enumerate}[label=(\Roman*)]
\item\label{rI} $\vec C$ is weakly coherent and $\otp(C_\alpha)\le\lambda$ for every $\alpha<\lambda^+$.
\end{enumerate}

Jensen's theorem was extended by Krueger \cite{MR3078820},
who proved that for every regular uncountable cardinal $\kappa$, there is a special $\kappa$-Aronszajn tree
iff there is a $C$-sequence $\vec C=\langle C_\alpha\mid\alpha<\kappa\rangle$ satisfying the following:\footnote{Strictly speaking, \cite[Definition~1.1]{MR3078820} is concerned with the existence of a $C$-sequence \emph{over a club} in $\kappa$. Extending it to get a $C$-sequence as in \ref{rII} is a trivial task.}
\begin{enumerate}[label=(\Roman*),resume]
\item\label{rII} $\vec C$ is weakly coherent and $\otp(C_\alpha)<\alpha$ for club many $\alpha<\kappa$.
\end{enumerate}

For our purpose, a sequence as in \ref{rII} is problematic, being in conflict with the fact that any witness $\vec C$
to a strong enough instance of the proxy principle must have stationarily many $\alpha$'s for which $|C_\alpha|=|\alpha|$,\footnote{See the argument of \cite[Remarks~3.16(1)]{paper32}.}
which, at the level of an inaccessible, means that stationarily many $\alpha<\kappa$ must satisfy $\otp(C_\alpha)=\alpha$.
The mitigation here comes from the work in \cite{paper71},
which pinpointed the impact of features of $\vec C=\langle C_\alpha\mid\alpha<\kappa\rangle$ on the corresponding canonical tree $T(\rho^{\vec C}_0)$
arising from walks on ordinals. It turned out that $T(\rho^{\vec C}_0)$ is a special $\kappa$-Aronszajn tree provided that the following holds:
\begin{enumerate}[label=(\Roman*),resume]
\item\label{rIII} $\vec C$ is weakly coherent and club many $\alpha<\kappa$ satisfy $\otp(C_\beta\cap\alpha)<\alpha$ for every $\beta\in\kappa\setminus\{\alpha\}$.
\end{enumerate}

Now, Clause~(1) of Theorem~\ref{thmb} establishes that a strong form of \ref{rIII} is compatible with having some thin stationary set of $\alpha$'s for which $|C_\alpha|=|\alpha|$,
and Clause~(2) shows that this subtle variation is indeed sufficient.
In a similar way, Clause~(1) solves a question of Shalev raised in the opening paragraph of \cite[\S4.1]{MR4833803}.
We find it interesting that while walks on ordinals and the recursive method of \cite{paper22}
provide two unrelated techniques for canonically producing trees out of $C$-sequences, a discovery concerning one of the methods can lead to a corresponding discovery about the other one.

\subsection{Notation and conventions}\label{nandc}
By an \emph{inaccessible} we mean a regular uncountable limit cardinal.
For a cardinal $\kappa$, we denote by $H_\kappa$ the collection of all sets of hereditary cardinality less than $\kappa$ (see Section IV.6 and Definition~III.3.4 of \cite{MR597342}).

\subsection{Organization of this paper}
In Section~\ref{secthma}, we provide some preliminaries on the interval topology and we prove Theorem~\ref{thma}.
To make this section more accessible, we have ensured that it is self-contained,
focuses on $\aleph_1$, assumes no familiarity with complicated background on tree constructions nor requires deep background in topology.
The section is concluded with a short discussion on variations of Theorem~\ref{thma}, also noting that an
$\aleph_1$-tree $\mathbf T$ for which $(X_{\mathbf T})^2$ is not cmc cannot be obtained on the grounds of $\zfc$ alone.

In Section~\ref{section:special}, we introduce the proxy principle $\p_{<}(\kappa,2,{\sq},\kappa)$ and prove its consistency, as promised in Clause~(1) of Theorem~\ref{thmb}.

In Section~\ref{section:thmb2}, we give our application of the new proxy principle, as promised in Clause~(2) of Theorem~\ref{thmb}.

\section{Theorem~\ref{thma}}\label{secthma}
To make this section accessible to a wide audience, we avoid unnecessary abstractions throughout.
In particular, we write $\Lambda$ for the set of infinite countable limit ordinals, and
we agree here on the following concrete implementation of trees, as follows.\footnote{To clarify, if $T$ is a tree in the sense of Definition~\ref{defn-tree}, then $\mathbf T:=(T,{\s})$ is a tree in the abstract sense,
and the height of any node $x\in T$ is nothing but $\dom(x)$.}
\begin{defn}\label{defn-tree}
A \emph{tree} is a subset $T$ of ${}^{<\omega_1}w$ for some countable set $w$ that is downward closed, i.e.,
for every $t\in T$ and every $\alpha<\dom(t)$, $t\restriction \alpha$ is as well in $T$.
$\h(T)$ stands for the least $\alpha\le\omega_1$ for which $T_\alpha:=T\cap{}^\alpha w$ is empty.
\end{defn}

Every two comparable nodes $x\subsetneq y$ of a tree $T$ give rise to an interval as follows: $(x,y]:=\{ z\in T\mid x\subsetneq z\s y\}$.
The \emph{interval topology} on $T$ has, as its basic open sets, the intervals $(x,y]$ for all $x,y\in T$ with $x\subsetneq y$, as well as the singleton $\{\emptyset\}$.
We denote by $X_T$ the outcome topological space.
Note that for every $t\in T$, $t$ is an isolated point of $X_T$ iff $\dom(t)\notin\Lambda$.
Also note that for every nonisolated $y\in T$, we have $\bigcup\{x\in T\mid x\subsetneq y\}=y$, and hence $X_T$ is a Hausdorff space.

\begin{defn}
\begin{enumerate}
\item A tree $T$ is an \emph{$\aleph_1$-tree} iff $\h(T)=\omega_1$, and $T_\alpha$ is countable for every $\alpha<\omega_1$.
\item A subset $A$ of a tree $T$ is an \emph{antichain} iff for every pair $x\s y$ of nodes from $A$, we have $x=y$.
\item An $\aleph_1$-tree $T$ is \emph{almost-Souslin} iff
for every antichain $A\s T$, the set $\{ \dom(x)\mid x\in A\}$ is nonstationary in $\omega_1$.
\item A tree $T$ is \emph{$\mathbb R$-embeddable} iff there exists a map $c:T\rightarrow\mathbb R$ such that for every pair $x\subsetneq y$ of nodes of $T$, $c(x)<c(y)$.
\end{enumerate}
\end{defn}

In \cite[Theorem~4.4]{Sh:86}, Devlin and Shelah constructed from $\diamondsuit^*(\omega_1)$ an $\mathbb R$-embeddable almost-Souslin $\aleph_1$-tree $T$ for which $X_T$ is not normal.
In \cite[Theorem~6]{MR623439}, Hanazawa constructed from $\diamondsuit^*(\omega_1)$ an $\mathbb R$-embeddable almost-Souslin $\aleph_1$-tree $T$ for which $X_T$ is normal.

\begin{fact}[{\cite[Theorems 3.3 and 4.1]{Sh:86}}]\label{devlinshelah} Suppose that $T$ is an $\aleph_1$-tree.
\begin{enumerate}[label=\textup{(\arabic*)}]
\item $T$ is almost-Souslin iff $X_T$ is \emph{collectionwise Hausdorff},
that is, iff for every closed discrete $A\s X_T$,
there exists a pairwise disjoint system $\langle O_x\mid x\in A\rangle$
such that, for every $x\in A$, $O_x$ is an open neighborhood of $x$.
\item The following condition on $T$, called \emph{property $\gamma$}, which clearly implies that $T$ is almost-Souslin, also implies that $X_T$ is normal: for every antichain $A\s T$,
there exists a club $D\s\omega_1$ such that $\bigcup_{\alpha\in\omega_1\setminus D}T_\alpha$
contains a closed neighborhood of $A$.
\end{enumerate}
\end{fact}
\begin{remark} The question of whether $X_T$ for a given $\aleph_1$-tree $T$ is normal goes back to Jones' work \cite{jones1966remarks} on \emph{Moore spaces} \cite{MR150722}.
The proof of \cite[Theorem~4.2]{Sh:86} shows that if $\diamondsuit(S)$ holds for every stationary $S\s\omega_1$,
then an $\aleph_1$-tree $T$ satisfies property $\gamma$ iff $X_T$ is normal.
\end{remark}
\begin{defn}[Two types of square]\label{square-def} Let $T$ be a tree.
\begin{itemize}
\item $(X_T)^2$ stands for the topological product space $X_T\times X_T$;
\item $T^2$ stands for the set $\{ (x_0,x_1)\in T\times T\mid \dom(x_0)=\dom(x_1)\}$,
which we typically equip with the ordering $\subseteq^2$ defined via $(x_0,x_1)\subseteq^2(y_0,y_1)$ iff $x_0\s y_0$ and $x_1\s y_1$.
\end{itemize}
\end{defn}
\begin{remark} The poset $(T^2,{\subseteq^2})$ is a tree in the abstract sense,
and the height of a node $(x_0,x_1)\in T^2$ is nothing but $\dom(x_0)$.
A subset $A$ of $T^2$ is a \emph{$\s^2$-antichain} iff for every pair $x\s^2 y$ of nodes from $A$, we have $x=y$.
\end{remark}

The next fact is standard. We include a proof for completeness.

\begin{prop}\label{prop26}
\begin{enumerate}[label=\textup{(\arabic*)}]
\item $T^2$ is a closed subspace of $(X_T)^2$;
\item Every $\s^2$-antichain in $T^2$ is a closed discrete subspace of $(X_T)^2$.
\end{enumerate}
\end{prop}
\begin{proof} (1) Given $(x_0,x_1)\in(X_T)^2\setminus T^2$,
we shall find an open neighborhood of $(x_0,x_1)$ disjoint from $T^2$.
Fix $i<2$ such that $\dom(x_i)<\dom(x_{1-i})$. Set $I_i:=[\emptyset,x_i]$,\footnote{For $x\in T$, write $[\emptyset,x] := \{\emptyset\}\cup(\emptyset,x]$, which is an open neighborhood of $x$ in $X_T$.}
and $I_{1-i}:=(x_{1-i}\restriction \dom(x_i),x_{1-i}]$. Then $U:=I_0\times I_1$ is an open neighborhood as sought.

(2) Given a $\s^2$-antichain $A\s T^2$, to show it is closed discrete, let $(x_0,x_1)\in(X_T)^2$ be given, and we shall find an open neighborhood $U$ of $(x_0,x_1)$ such that $A\cap U\s\{(x_0,x_1)\}$.
By Clause~(1), we may assume that $(x_0,x_1)\in T^2$.
Thus, consider the set $A':=\{ (y_0,y_1)\in A\mid y_0\subsetneq x_0\ \&\allowbreak\ y_1\subsetneq x_1\}$.
Every two pairs in $A'$ are compatible elements of $(T^2,{\s^2})$. But $A'$ is a subset of a $\s^2$-antichain, and hence there are only two possibilities:

$\br$ $A'$ is empty. In this case, $U:=\prod_{i<2}[\emptyset,x_i]$ is an open neighborhood as sought.

$\br$ $A'$ is a singleton, say, $A'=\{(y_0,y_1)\}$. In this case, $U:=\prod_{i<2}(y_i,x_i]$ is an open neighborhood as sought.
\end{proof}

\subsection{Trees embeddable to the reals}
For the scope of Section~\ref{secthma}, we define a map $c:{}^{<\omega_1}\mathbb Q\rightarrow\mathbb{R}\cup\{\infty\}$ via
$$c(x):=
\begin{cases}
0,&\text{if }x=\emptyset;\\
\sup(\im(x)),&\text{otherwise}.
\end{cases}$$

\begin{defn} $\T$ denotes the collection of all $x\in {}^{<\omega_1}\mathbb Q$
for which $\langle c(x\restriction \beta)\mid \beta\le\dom(x)\rangle$ is a strictly increasing sequence of real numbers.
\end{defn}

The following facts are readily checked:
\begin{prop}\label{359}
\begin{enumerate}
\item $\T$ is a tree (in the sense of Definition~\ref{defn-tree}).
\item For every $x\in\T$ with $\dom(x)=\alpha+1$ a successor, $c(x)=x(\alpha)$.
\item $c\restriction\T$ is a strictly increasing map from $(\T,{\s})$ to $(\mathbb R,{\le})$.
\item For every strictly $\subseteq$-increasing sequence $\langle t_n \mid n<\omega \rangle$ of elements of $\T$,
if $\sup\{c(t_n) \mid n<\omega \}\neq\infty$,
then the unique limit of the sequence, $t:= \bigcup\{t_n \mid n<\omega\}$, is also in $\T$,
and $c(t) = \sup\{c(t_n) \mid n<\omega \}$. \qed
\end{enumerate}
\end{prop}
\begin{cor} Every tree $T\s\T$ is $\mathbb R$-embeddable.\footnote{The poset $(\T,\s)$ is essentially the same as the poset $(\sigma\mathbb Q,{\leqdot})$
from \cite[p.~245]{MR776625}, in the sense of possessing a universal feature for $\mathbb R$-embeddable trees as in \cite[Proposition~3]{MR2783794}.}
\qed
\end{cor}
\begin{defn}
For all $x\in\T$ and $q\in{\mathbb Q}$, denote
$$U(x,q):=\{y\in \T\mid x\subsetneq y\ \&\ c(y)<c(x)+q\}.$$
\end{defn}

Evidently, for every tree $T\s\T$, for all $x\in T$ and $q\in{\mathbb Q}$, the set $U(x,q)\cap T$ is open in $X_T$.

\subsection{Elevators and friends} We now introduce a few key concepts that will aid in our upcoming construction.

\begin{defn}[Tree-coarsening]\label{coarsening} For a tree $T$, we say that a partial ordering $\LE$ on $T^2$ is a \emph{tree-coarsening} of $\s^2$ iff for all $(x_0,x_1),(y_0,y_1)\in T^2$:
\begin{itemize}
\item if $(x_0,x_1)\LE(y_0,y_1)$, then $(x_0,x_1)\s^2(y_0,y_1)$;
\item if $(x_0,x_1)\LE(y_0,y_1)$, then $(x_1,x_0)\LE(y_1,y_0)$;
\item if $(x_0,x_1)\LE(y_0,y_1)$, then for every $\beta$ with $\dom(x_0)<\beta<\dom(y_0)$,
$(x_0,x_1)\LE(y_0\restriction\beta,y_1\restriction\beta)\LE(y_0,y_1)$.
\end{itemize}
\end{defn}

\begin{defn}[$q$-elevator]\label{n-elevator} For a tree $T\s\T$, a tree-coarsening ${\LE}$ of $\subseteq^2$, ordinals $\beta<\alpha<\h(T)$, and $q\in{\mathbb Q}$,
a function $e$ from a subset of $T_\beta$ to $T_\alpha$ is said to be a \emph{$q$-elevator} (with respect to $\LE$) iff:
\begin{itemize}
\item for every $x\in\dom(e)$, $e(x)\in U(x,q)$, and
\item for every $(x_0,x_1)\in(\dom(e))^2$, $(x_0,x_1)\LE (e(x_0),e(x_1))$.
\end{itemize}
\end{defn}

\begin{remark}\label{elevator-composition} The composition of a $p$-elevator with a $q$-elevator is an $r$-elevator for every $r\geq p+q$.
\end{remark}
\begin{notation} $\bar{\mathbb Q}$ stands for $\mathbb Q\cap(0,1)$.
\end{notation}

\begin{defn}[Coordination]\label{Omega_coordinated}
For a tree $T\s\T$, a tree-coarsening ${\LE}$ of $\subseteq^2$,
and ordinals $\beta<\alpha<\h(T)$, we say that $T_\beta$ and $T_\alpha$ are \emph{coordinated} (with respect to $\LE$) iff for every $q\in\bar{\mathbb Q}$, the following three hold:
\begin{enumerate}[start=0,label=\textup{(\arabic*)}]
\item\label{cord0} For every finite $W\s T_\alpha$, there exists a $q$-elevator $e:T_\beta\rightarrow T_\alpha$ such that $\im(e)\cap W=\emptyset$.
\item\label{cord1} For every $x\in T_\beta$, every $y\in U(x,q)\cap T_\alpha$ and every finite set $W\s T_\alpha\setminus \{y\}$,
there exists a $q$-elevator $e:T_\beta\rightarrow T_\alpha$ such that:
\begin{itemize}
\item $e(x)=y$, and
\item $\im(e)\cap W=\emptyset$.
\end{itemize}
\item\label{cord2} For every pair $(x_0,x_1)\in T_\beta\times T_\beta$ with $x_0\neq x_1$, for every pair $(y_0,y_1)\in(U(x_0,q)\cap T_\alpha)\times(U(x_1,q)\cap T_\alpha)$ such that $(x_0,x_1)\LE (y_0,y_1)$,
and every finite set $W\s T_\alpha\setminus \{y_0,y_1\}$, there exists a $q$-elevator $e:T_\beta\rightarrow T_\alpha$ such that:
\begin{itemize}
\item $e(x_0)=y_0$,
\item $e(x_1)=y_1$, and
\item $\im(e)\cap W=\emptyset$.
\end{itemize}
\end{enumerate}
\end{defn}
\begin{remark}
It follows from Clause~\ref{cord0} above that for all $x\in T_\beta$ and $q\in\mathbb Q$, there is a $y$ in $U(x,q)\cap T_\alpha$.
\end{remark}

\begin{lemma}\label{coordination-transitive} Suppose:
\begin{itemize}
\item $T\s\T$ is a tree;
\item $\LE$ is a tree-coarsening of $\s^{2}$;
\item $\gamma<\beta<\alpha<\h(T)$ are ordinals;
\item $T_\gamma$ and $T_\beta$ are coordinated, and $T_\beta$ and $T_\alpha$ are coordinated;
\item for all $(x_0,x_1)\in T_\beta\times T_\beta$ and $(y_0,y_1)\in T_\alpha\times T_\alpha$, $(x_0,x_1)\LE(y_0,y_1)$ iff $(x_0,x_1)\s^2(y_0,y_1)$.
\end{itemize}

Then $T_\gamma$ and $T_\alpha$ are coordinated.
\end{lemma}
\begin{proof} Let $q\in\bar{\mathbb Q}$.
We go over the clauses of Definition~\ref{Omega_coordinated}, keeping in mind Remark~\ref{elevator-composition}:
\begin{enumerate}[start=0]
\item Consider a given finite $W\s T_\alpha$.
As $T_\gamma$ and $T_\beta$ are coordinated,
we may fix a $\frac{q}{2}$-elevator $e_0:T_\gamma\rightarrow T_\beta$.
As $T_\beta$ and $T_\alpha$ are coordinated,
we may fix a $\frac{q}{2}$-elevator $e_1:T_\beta\rightarrow T_\alpha\setminus W$.
Then $e:=e_1\circ e_0$ is a $q$-elevator from $T_\gamma$ to $T_\alpha$ satisfying that $\im(e)\cap W=\emptyset$.

\item Let $x\in T_\gamma$, $z\in U(x,q)\cap T_\alpha$ and a finite set $W\s T_\alpha\setminus \{z\}$
be given; we need to find a $q$-elevator $e:T_\gamma\rightarrow T_\alpha\setminus W$ such that $e(x)=z$.

As $c(z)-c(x)<q$, we may find some $\varepsilon\in\bar{\mathbb Q}$ such that
\begin{enumerate}[label=(\roman*)]
\item\label{457} $c(z)-c(x)+2\varepsilon<q$.
\end{enumerate}

Set $y:=z\restriction \beta$, and then pick $q_0,q_1\in\bar{\mathbb Q}$ such that
\begin{enumerate}[label=(\roman*),resume]
\item $c(y)-c(x)<q_0<c(y)-c(x)+\varepsilon$, and
\item\label{465} $c(z)-c(y)<q_1<c(z)-c(y)+\varepsilon$.
\end{enumerate}

As $y\in U(x,q_0)\cap T_\beta$, we may fix a $q_0$-elevator $e_0:T_\gamma\rightarrow T_\beta$ such that $e_0(x)=y$.
As $z\in U(y,q_1)\cap T_\alpha$ and $W\s T_\alpha\setminus \{z\}$, we may fix a $q_1$-elevator $e_1:T_\beta\rightarrow T_\alpha\setminus W$ such that $e_1(y)=z$.
By \ref{457}--\ref{465}, $q_0+q_1<q$, and hence $e:=e_1\circ e_0$ is altogether a $q$-elevator as sought.

\item Let $\{x_0,x_1\}\in [T_\gamma]^2$, $(z_0,z_1)\in (U(x_0,q)\cap T_\alpha)\times(U(x_1,q)\cap T_\alpha)$ such that $(x_0,x_1)\LE (z_0,z_1)$,
and a finite set $W\s T_\alpha\setminus \{z_0,z_1\}$ be given;
we need to find a $q$-elevator $e:T_\gamma\rightarrow T_\alpha\setminus W$ such that
$e(x_0)=z_0$ and $e(x_1)=z_1$.

Set $y_0:=z_0\restriction \beta$ and $y_1:=z_1\restriction \beta$.
Fix a large enough $q_0\in \mathbb Q\cap (0,q)$ such that
$(y_0,y_1)\in (U(x_0,q_0)\cap T_\beta)\times(U(x_1,q_0)\cap T_\beta)$.
By Definition~\ref{coarsening}, $(x_0,x_1)\LE(y_0,y_1)$.
Thus, as $T_\gamma$ and $T_\beta$ are coordinated, we may fix a $q_0$-elevator $e_0:T_\gamma\rightarrow T_\beta$ such that $e_0(x_j)=y_j$ for every $j<2$.
Set $q_1:=q-q_0$.
As $T_\beta$ and $T_\alpha$ are coordinated, we may fix a $q_1$-elevator $e_1:T_\beta\rightarrow T_\alpha\setminus W$.
It is clear that $e_1\circ e_0$ is a $q$-elevator whose image is disjoint from $W$,
but we did not secure that $x_j$ goes to $z_j$ for every $j<2$.
To this end, using the fact that $\{z_0,z_1\}\cap W=\emptyset$, we define a map $e:T_\gamma\rightarrow T_\alpha\setminus W$ via
$$e(t):=
\begin{cases}
z_0,&\text{if }t=x_0;\\
z_1,&\text{if }t=x_1;\\
e_1(e_0(t)),&\text{otherwise}.
\end{cases}$$

It is evident that $e(t)\in U(t,q)$ for every $t\in T_\gamma$.
To see that the second bullet point of Definition~\ref{n-elevator} holds as well, let $(t_0,t_1)\in T_\gamma\times T_\gamma$ be given.
A moment's reflection makes it clear that for every $t\in T_\gamma$, $e(t)\restriction\beta=e_0(t)$. Therefore, $(t_0,t_1)\LE(e(t_0)\restriction\beta,e(t_1)\restriction\beta)$.
In addition, $(e(t_0)\restriction\beta,e(t_1)\restriction\beta)\s^2(e(t_0),e(t_1))$ holds trivially.
Then, by the hypotheses,
furthermore, $(e(t_0)\restriction\beta,e(t_1)\restriction\beta)\LE(e(t_0),e(t_1))$.
Altogether, $(t_0,t_1)\LE(e(t_0),e(t_1))$. \qedhere
\end{enumerate}
\end{proof}

\subsection{The construction} Our upcoming applications of the principle $\diamondsuit^*(\omega_1)$ and its consequence $\diamondsuit(\omega_1)$
are encapsulated by the following two easy facts. The first derives a particular ladder system,
and the second is a more versatile formulation
of $\diamondsuit(\omega_1)$ motivated by the fact that $\T$ is a subset of $H_{\omega_1}$.
Here, instead of predicting initial segments of given subsets of $\omega_1$,
we predict the extent seen by a countable elementary submodel of given subsets of $H_{\omega_1}$.

\begin{fact}[special case of {\cite[Theorem~4.35]{paper23}}]\label{proxystar}
$\diamondsuit^*(\omega_1)$ implies that there is a sequence $\vec C=\langle C_\alpha\mid\alpha<\omega_1\rangle$ satisfying the following two:
\begin{itemize}
\item for every $\alpha\in\Lambda$, $C_\alpha$ is a cofinal subset of $\alpha$ of order-type $\omega$;
\item for every uncountable $B\s\omega_1$, there are club many $\alpha<\omega_1$ such that $\sup(C_\alpha\cap B)=\alpha$.
\end{itemize}
\end{fact}

\begin{fact}[special case of {\cite[Lemma~2.2]{paper22}}]\label{diamond_hw1}
$\diamondsuit(\omega_1)$ is equivalent to the existence of a partition $\langle R_i \mid i < \omega_1 \rangle$ of $\omega_1$
and a sequence $\langle \Omega_\beta \mid \beta < \omega_1 \rangle$ of countable sets
such that for all $p\in H_{\omega_2}$, $i<\omega_1$, and $\Omega \subseteq H_{\omega_1}$,
there exists a countable elementary submodel $\mathcal M\prec H_{\omega_2}$ such that:
\begin{itemize}
\item $p\in \mathcal M$;
\item $\beta := \mathcal M\cap\omega_1$ is an ordinal in $R_i$;
\item $\mathcal M\cap \Omega=\Omega_\beta$.
\end{itemize}
\end{fact}

We now arrive at the main result of this section, namely, Theorem~\ref{thma}.
\begin{thm}\label{thm211}
Suppose that $\diamondsuit^*(\omega_1)$ holds. Then there is
an $\mathbb R$-embeddable almost-Souslin $\aleph_1$-tree $T$ such that $X_T$ is perfectly normal, but $(X_T)^2$ is not cmc.
\end{thm}
\begin{proof} As $\ch$ holds, let $\phi:\omega_1\leftrightarrow H_{\omega_1}$ be any bijection,
and let $\lhd_{\omega_1}$ be the induced well-ordering of $H_{\omega_1}$.
As $\diamondsuit(\omega_1)$ holds, let $\langle R_i\mid i<\omega_1\rangle$ and $\langle \Omega_\beta\mid\beta<\omega_1\rangle$ be given by Fact~\ref{diamond_hw1}.
Let $\pi:\omega_1\rightarrow\omega_1$ be the unique function satisfying $\alpha\in R_{\pi(\alpha)}$ for all $\alpha<\omega_1$.
Then, set $\psi:=\phi\circ \pi$.
Finally, as $\diamondsuit^*(\omega_1)$ holds, let $\langle C_\alpha\mid\alpha<\omega_1\rangle$ be given by Fact~\ref{proxystar}.
We may assume that $\min(C_\alpha)=1$ for every $\alpha\in\Lambda$.

We shall construct an $\aleph_1$-tree $T\s\T$ along with a subset $E\s\omega_1$
and a $\s^2$-antichain $\langle a_\epsilon\mid \epsilon\in E\rangle\in\prod_{\epsilon\in E}(T_\epsilon)^2$.

This $\s^2$-antichain will induce a tree-coarsening $\LE$ of $\s^2$ defined by letting $(x_0,x_1)\LE (y_0,y_1)$ iff:
\begin{enumerate}[label=(\(\protect\AlephOrBeth\))]
\item $(x_0,x_1)\s^2 (y_0,y_1)$, and
\item\label{beth} for every $j<2$, if $a_\epsilon\nsubseteq^2 (x_j,x_{1-j})$ for every $\epsilon\in E\cap(\dom(x_0)+1)$,
then $a_\epsilon\nsubseteq^2 (y_j,y_{1-j})$ for every $\epsilon\in E\cap(\dom(y_0)+1)$.
\end{enumerate}

The construction of the tree $T$ is by recursion on $\alpha<\omega_1$, where at stage $\alpha$, we determine a countable set $T_\alpha$ (the $\alpha^{\text{th}}$ level of $T$),
decide whether $\alpha$ is in $E$, and if it is, determine also an element $a_\alpha\in(T_\alpha)^2$.

For all $\beta<\alpha<\omega_1$, we will ensure that $T_\beta$ and $T_\alpha$ are coordinated.

For any ordinal $\alpha<\omega_1$ such that $\langle T_\beta\mid\beta<\alpha\rangle$ has already been determined
and for every $B\s\alpha$, we shall write $T\restriction B:=\bigcup_{\beta\in B}T_\beta$,
and we observe that the restriction of $\LE$ to $(T\restriction\alpha)^2$ is determined by the initial segment
$\langle a_\epsilon\mid \epsilon\in E\cap\alpha\rangle$ of the eventual $\s^2$-antichain.
We also observe the following:

\begin{claim}\label{claim2141} For all $\alpha<\omega_1$ and $(x_0,x_1),(y_0,y_1)$ in $(T\restriction \alpha)^2$,
if the interval $(\dom(x_0),\alpha)$ is disjoint from $E$, then
$$(x_0,x_1)\LE (y_0,y_1) \iff (x_0,x_1)\s^2 (y_0,y_1).\eqno\qed$$
\end{claim}

The preparations are over, and we now turn to the construction.
We start by setting $T_0:=\{\emptyset\}$ and deciding that $0\notin E$.

Next, given $\alpha<\omega_1$ such that $T_\alpha$ has already been defined, we set
\begin{align*}
T_{\alpha+1}:={}&\{ x \in {}^{\alpha+1}\mathbb{Q} \mid (x\restriction\alpha)\in T_\alpha \}\cap\T\\
={}&\{t^\smallfrown\langle q\rangle\mid t\in T_\alpha,~ q\in\mathbb Q, ~ q>c(t)\},
\end{align*}
and decide not to include $\alpha+1$ in $E$.
As $T_\alpha$ is a countable set, so is $T_{\alpha+1}$.

\begin{claim} For every $\beta<\alpha+1$, $T_\beta$ and $T_{\alpha+1}$ are coordinated.
\end{claim}
\begin{proof} As $\alpha+1\notin E$, Claim~\ref{claim2141} implies that
for all $(x_0,x_1)\in T_\alpha\times T_\alpha$ and $(y_0,y_1)\in T_{\alpha+1}\times T_{\alpha+1}$, $(x_0,x_1)\LE(y_0,y_1)$ iff $(x_0,x_1)\s^2(y_0,y_1)$.
Thus, by Lemma~\ref{coordination-transitive} and the induction hypothesis, it suffices to prove that $T_\alpha$ and $T_{\alpha+1}$ are coordinated.
To this end, let $q\in\bar{\mathbb Q}$ be given, and we shall go over the three clauses of Definition~\ref{Omega_coordinated}:
\begin{enumerate}[start=0]
\item Given a finite $W\s T_{\alpha+1}$, set
$$r:=\min\{q,c(w)-c(w\restriction \alpha)\mid w\in W\},$$
and then fix a system $\langle q_t\mid t\in T_\alpha\rangle$ of rational numbers
such that $c(t)<q_t<c(t)+r$ for every $t\in T_\alpha$.
Define a map $e:T_\alpha\rightarrow T_{\alpha+1}$ via $e(t):=t{}^\smallfrown \langle q_t\rangle$.
As $\alpha+1\notin E$ and $r\leq q$, $e$ is a $q$-elevator.
For every $w \in W$, $c(e(w\restriction\alpha)) = q_{w\restriction\alpha} < c(w\restriction\alpha)+r \leq c(w)$,
so that $e(w\restriction\alpha) \neq w$. Thus, it is also the case that $\im(e)\cap W=\emptyset$.

\item Let $x\in T_\alpha$, $y\in U(x,q)\cap T_{\alpha+1}$ and a finite set $W\s T_{\alpha+1}\setminus \{y\}$
be given; we need to find a $q$-elevator $e:T_\alpha\rightarrow T_{\alpha+1}\setminus W$ such that $e(x)=y$.
Obtain $r$ and $\langle q_t\mid t\in T_\alpha\rangle$ as in Clause~(1), and then define a map $e:T_\alpha\rightarrow T_{\alpha+1}$ via
$$e(t):=\begin{cases}
y,&\text{if }t=x;\\
t^\smallfrown \langle q_t\rangle,&\text{otherwise}.
\end{cases}$$

As $\alpha+1\notin E$, $e$ is a $q$-elevator. As $y\notin W$, it is also the case that $\im(e)\cap W=\emptyset$, just as in Clause~(1).

\item Let $\{x_0,x_1\}\in [T_\alpha]^2$, $(y_0,y_1)\in (U(x_0,q)\cap T_{\alpha+1})\times(U(x_1,q)\cap T_{\alpha+1})$ such that $(x_0,x_1)\LE (y_0,y_1)$,
and a finite set $W\s T_{\alpha+1}\setminus \{y_0,y_1\}$ be given;
we need to find a $q$-elevator $e:T_\alpha\rightarrow T_{\alpha+1}\setminus W$ such that
$e(x_0)=y_0$ and $e(x_1)=y_1$.
Obtain $r$ and $\langle q_t\mid t\in T_\alpha\rangle$ as in Clause~(1),
and then define a map $e:T_\alpha\rightarrow T_{\alpha+1}$ via
$$e(t):=\begin{cases}
y_0,&\text{if }t=x_0;\\
y_1,&\text{if }t=x_1;\\
t^\smallfrown \langle q_t\rangle,&\text{otherwise}.
\end{cases}$$
Then, $e$ is a $q$-elevator as sought.\qedhere
\end{enumerate}
\end{proof}
Now, fix a given $\alpha\in\Lambda$ such that $\langle T_\beta \mid \beta<\alpha \rangle$ and $\langle a_\epsilon\mid \epsilon\in E\cap\alpha\rangle$ have already been successfully defined.
In particular, for all $\gamma<\beta<\alpha$, $T_\gamma$ and $T_\beta$ are coordinated.

Consider the collection $\mathcal B^\alpha:=\{ t \in {}^\alpha\mathbb{Q} \mid \forall\beta<\alpha\, (t\restriction\beta \in T_\beta) \}$
of all cofinal branches through $T\restriction\alpha$.
For each $x\in T\restriction C_\alpha$
we shall carefully identify some element $\mathbf b_x^{\alpha}\in \mathcal B^\alpha\cap\T$ with $x\s \mathbf b_x^{\alpha}$,
and we shall then define the $\alpha^{\text{th}}$ level of the tree to be $$T_\alpha:=\{\mathbf b_x^{\alpha}\mid x\in T\restriction C_\alpha\}.$$

To this end, we denote by $\langle\beta_n\mid n<\omega\rangle$ the increasing enumeration of $C_\alpha$,
and we plan to construct, recursively, a sequence $\langle (e_n,q_n)\mid n<\omega\rangle$,
where, for every $n<\omega$, $e_n:T_{\beta_n}\rightarrow T_{\beta_{n+1}}$ is a $q_{n+1}$-elevator.\footnote{Strictly speaking, the notation should have been $\beta_n^\alpha$, $e_n^\alpha$ and $q_n^\alpha$,
but we suppress the superscript for brevity as we will always be working in the context of a fixed value of $\alpha$.}
For each $x\in T\restriction C_\alpha$, those $e_n$'s will determine $\mathbf{b}_x^{\alpha}$ as the limit $\bigcup\im(b_x^{\alpha})$ of the unique $\s$-increasing sequence $b_x^{\alpha}$ satisfying the following three properties:
\begin{enumerate}[label=(\roman*)]
\item\label{clausei} $\dom(b^\alpha_x)=\{\beta_n\mid n<\omega\}\setminus \dom(x)$;
\item\label{clauseii} $b_x^{\alpha}(\dom(x))=x$;
\item\label{clauseiii} for every $n<\omega$ with $\beta_n\ge\dom(x)$, $b_x^{\alpha}(\beta_{n+1})=e_n(b_x^\alpha(\beta_n))$.
\end{enumerate}
In particular, for the unique $k<\omega$ such that $\dom(x)=\beta_k$,\footnote{Recall that $\dom(x)$ is the height of $x$ in the tree, and that $x\in T\restriction C_\alpha$.}
it would be the case that $b_x^{\alpha}(\beta_{n})\in T_{\beta_n}$ for every $n\in\omega\setminus k$.

As for the $q_n$'s, we set $q_0:=1$, and announce at the outset that for every $n<\omega$,
there will be three possible cases; in the first two, we shall let $q_{n+1}:=\frac{q_n}{2}$,
and in the third, we shall let $q_{n+1}:=\frac{q_n}{8}$. Consequently, for every $n<\omega$,
if we fall into the first two cases, then $(\sum_{m=n+1}^\infty q_m)\le q_n$,
and otherwise, $(\sum_{m=n+1}^\infty q_m)\le \frac{q_n}{4}$.
In particular, $\lim_{n\rightarrow\infty}q_n=0$.

We now turn to the actual construction.
Suppose $n<\omega$ is such that the sequence $\langle (e_k,q_{k+1})\mid k<n\rangle$ has already been defined.
In particular, Clauses \ref{clausei}--\ref{clauseiii} have already determined $b_x^\alpha(\beta_n)$ for every $x\in T\restriction(C_\alpha\cap\beta_{n+1})=\bigcup_{k\le n}T_{\beta_k}$.
As announced, the definition of $(e_n,q_{n+1})$ is divided into three cases. They read as follows:
\begin{description}
\item[Case~I] Suppose that all of the following hold:
\begin{itemize}
\item $\Omega_{\beta_{n+1}}$ is a subset of $T\restriction \beta_{n+1}$,
\item $w:=\psi(\beta_{n+1})$ is an element of $T\restriction(C_\alpha\cap \beta_{n+1})$,
so that, in particular, $b_w^\alpha(\beta_n)$ is an element of $T_{\beta_n}$, and
\item the set $Q_{n+1}^{\alpha}:=\{t\in U(b_w^{\alpha}(\beta_n),\frac{q_n}{2})\cap T_{\beta_{n+1}}\mid \exists a\in \Omega_{\beta_{n+1}}\, (a\s t)\}$ is nonempty.
\end{itemize}
In this case, set $q_{n+1}:=\frac{q_n}{2}$, and choose some $q_{n+1}$-elevator $e_n:T_{\beta_n}\rightarrow T_{\beta_{n+1}}$ such that $e_n(b_w^{\alpha}(\beta_n))=\min(Q^{\alpha}_{n+1},{\lhd_{\omega_1}})$,
which must exist by Clause~\ref{cord1} of coordination of $T_{\beta_n}$ and $T_{\beta_{n+1}}$.

\item[Case~II] Suppose that all of the following hold:
\begin{itemize}
\item $\Omega_{\beta_{n+1}}$ is a function from $(T\restriction\beta_{n+1})^2$ to $\omega$,
\item $k:=\psi(\beta_{n+1})$ is an element of $\omega$, and
\item the set $P^{\alpha}_{n+1}$ of all pairs $(t_1,t_2)\in (T_{\beta_{n+1}})^2$ satisfying all of the following is nonempty:
\begin{itemize}
\item $(t_1,t_2)\in U(b_{\langle1\rangle}^{\alpha}(\beta_n),\frac{q_n}{2})\times U(b_{\langle2\rangle}^{\alpha}(\beta_n),\frac{q_n}{2})$,\footnote{Recall that $\beta_0 = \min(C_\alpha)=1$,
so that $\langle1\rangle$ and $\langle2\rangle$ are in $T\restriction (C_\alpha\cap\beta_{n+1})$
and therefore, in particular, both $b_{\langle1\rangle}^\alpha(\beta_n)$ and $b_{\langle2\rangle}^\alpha(\beta_n)$ are in $T_{\beta_n}$.}
\item $(b_{\langle1\rangle}^{\alpha}(\beta_n),b_{\langle2\rangle}^{\alpha}(\beta_n))\LE (t_1,t_2)$, and
\item there is $\tau\in[\beta_n,\beta_{n+1})$ such that $\Omega_{\beta_{n+1}}(t_1\restriction \tau,t_2\restriction \tau)=k$.
\end{itemize}
\end{itemize}
In this case, set $(t_1,t_2):=\min(P^{\alpha}_{n+1},{\lhd_{\omega_1}})$,
$q_{n+1}:=\frac{q_n}{2}$, and choose some $q_{n+1}$-elevator $e_n:T_{\beta_n}\rightarrow T_{\beta_{n+1}}$ such that $e_n(b_{\langle j\rangle}^{\alpha}(\beta_n))=t_j$ for $j\in\{1,2\}$,
which must exist by Clause~\ref{cord2} of coordination of $T_{\beta_n}$ and $T_{\beta_{n+1}}$.
\item[Case~III] Otherwise.
In this case, set $q_{n+1}:=\frac{q_n}{8}$ and choose any $q_{n+1}$-elevator $e_n:T_{\beta_n}\rightarrow T_{\beta_{n+1}}$,
which must exist by Clause~\ref{cord0} of coordination of $T_{\beta_n}$ and $T_{\beta_{n+1}}$.
\end{description}

We record the following crucial features that follow from the above construction together with Definition~\ref{n-elevator}, Remark~\ref{elevator-composition}, and Proposition~\ref{359}(4):
\begin{claim}\label{cfeature} For all $n<m<\omega$ and $x,x'\in T\restriction(C_\alpha\cap\beta_{n+1})$, the following hold:
\begin{itemize}
\item $\mathbf{b}_x^\alpha\restriction\beta_n = b_x^\alpha(\beta_n)$;
\item $(b_x^\alpha(\beta_n),b_{x'}^\alpha(\beta_n))\LE (b_x^\alpha(\beta_m),b_{x'}^\alpha(\beta_m))$;
\item if $e_n$ was defined according to Case~III, then $\mathbf b_x^\alpha\in U(b_x^\alpha(\beta_n),\frac{q_n}{4})$;\footnote{The importance of $\frac{q_n}{4}$ will become clear in the proof of \hyperref[c218812]{Subclaim~\the\numexpr\getrefnumber{c21881}(2)}.}
otherwise, $\mathbf b_x^\alpha\in U(b_x^\alpha(\beta_n),{q_n})$. So, in either case, $\mathbf b_x^\alpha\in\T$. \qed
\end{itemize}
\end{claim}

Having completed the recursive construction of $\langle (e_n,q_n)\mid n<\omega\rangle$,
for each $x\in T\restriction C_\alpha$, the corresponding ascending sequence $b_x^\alpha$
and its limit $\mathbf b_x^{\alpha}:=\bigcup\im(b_x^{\alpha})$ have been completely determined, so we now set, as promised,
$$T_\alpha:=\{\mathbf b_x^{\alpha}\mid x\in T\restriction C_\alpha\},$$
which is a subset of $\T$ by the preceding claim. As $T\restriction\alpha$ is a countable set, so is $T_\alpha$.

\begin{claim}\label{tailf}
For every $y\in T_\alpha$, there are co-finitely many $m<\omega$ such that $y=\mathbf b_{y\restriction \beta_m}^\alpha$.
\end{claim}
\begin{proof} Given $y\in T_\alpha$, by definition of $T_\alpha$ there is some $x\in T\restriction C_\alpha$ such that $y=\mathbf b_x^{\alpha}$.
Choose any such $x$, and let $k<\omega$ be such that $\dom(x)=\beta_k$.
Consider any given $m\in[k,\omega)$; we will show that $\mathbf b_{y\restriction \beta_m}^\alpha=y$.\footnote{The same proof will show that $(y=\mathbf b_{y\restriction \beta_k}^\alpha)\implies(y=\mathbf b_{y\restriction \beta_{k+1}}^\alpha)$ for every $k<\omega$, but we shall not need that.}
Observe that $\mathbf b^\alpha_{y\restriction\beta_m}=\bigcup\{b_{y\restriction\beta_m}^\alpha(\beta_n)\mid m\le n<\omega\}$
and, since $k\leq m$ and the sequence $b_x^\alpha$ is $\s$-increasing,
$y=\mathbf{b}_x^\alpha=\bigcup\{b_{x}^\alpha(\beta_n)\mid k\le n<\omega\}
=\bigcup\{b_{x}^\alpha(\beta_n)\mid m\le n<\omega\}$.
Thus, it suffices to show that $b_{y\restriction \beta_m}^\alpha(\beta_n)=b_x^\alpha(\beta_n)$ for all $n\in[m,\omega)$.
We prove this by induction:
\begin{itemize}
\item For $n=m$, Clause~\ref{clauseii} (see page~\pageref{clauseii}) gives
$$b_{y\restriction \beta_m}^\alpha(\beta_m)=y\restriction\beta_m=\mathbf b_x^{\alpha}\restriction\beta_m=b_x^\alpha(\beta_m).$$
\item For every $n\in[m,\omega)$ such that $b_{y\restriction \beta_m}^\alpha(\beta_{n})=b_x^\alpha(\beta_{n})$,
as $\dom(x)=\beta_k \leq \beta_m=\dom(y\restriction\beta_m) \leq \beta_n$,
Clause~\ref{clauseiii} gives
$$b_{y\restriction \beta_m}^\alpha(\beta_{n+1})=e_{n}(b_{y\restriction \beta_m}^\alpha(\beta_{n}))=e_{n}(b_x^\alpha(\beta_{n}))=b_x^\alpha(\beta_{n+1}),$$
completing the induction and thereby proving the claim.
\qedhere
\end{itemize}
\end{proof}

At this point, we need to decide whether to include $\alpha$ in $E$, and if so, then also to determine the identity of $a_\alpha$.

If $\Omega_\alpha$ happens to be a function from $(T\restriction\alpha)^2$ to $\omega$, then consider the set
$$K_\alpha:=\{k<\omega\mid \sup\{\tau<\alpha\mid \Omega_\alpha(\mathbf b_{\langle1\rangle}^\alpha\restriction\tau,\mathbf b_{\langle2\rangle}^\alpha\restriction\tau)=k\}=\alpha\},$$
and in case $K_\alpha \neq \emptyset$, we include $\alpha$ in $E$ and set $a_\alpha:=(\mathbf b_{\langle1\rangle}^\alpha,\mathbf b_{\langle2\rangle}^\alpha)$.
Otherwise, we do not include $\alpha$ in $E$.

\begin{claim}\label{clanti} $\langle a_\epsilon\mid \epsilon\in E\cap(\alpha+1)\rangle$ is a $\s^2$-antichain.
\end{claim}
\begin{proof}
By the induction hypothesis, $\langle a_\epsilon\mid \epsilon \in E\cap\alpha\rangle$ is a $\s^2$-antichain.
Thus, it suffices to consider the case $\alpha\in E$ and prove that $a_\epsilon \nsubseteq^2a_\alpha$ for every $\epsilon\in E\cap\alpha$.
To this end, let $\epsilon\in E\cap\alpha$ be given.
Choose a large enough $m<\omega$ such that $\beta_m>\epsilon$. Appealing to Claim~\ref{cfeature} with $n:=0$, $x:=\langle1\rangle$ and $x':=\langle2\rangle$,
we obtain $(\langle1\rangle,\langle2\rangle) \LE (b_{\langle1\rangle}^\alpha(\beta_m), b_{\langle2\rangle}^\alpha(\beta_m))$.
By Clause~\ref{beth}, from $E\cap2 = \emptyset$, we obtain $a_\epsilon \nsubseteq^2 (b_{\langle1\rangle}^\alpha(\beta_m), b_{\langle2\rangle}^\alpha(\beta_m))$.
In particular, $a_\epsilon \nsubseteq^2 (\mathbf b_{\langle1\rangle}^\alpha,\mathbf b_{\langle2\rangle}^\alpha)$.
But the latter is equal to $a_\alpha$, so we are done.
\end{proof}

\begin{claim}\label{claim2145}
For every $\beta<\alpha$, $T_\beta$ and $T_\alpha$ are coordinated.
\end{claim}
\begin{proof} Before we start, for every $m<\omega$, we define a map $f_m:T_{\beta_m}\rightarrow T_\alpha$
via $f_m(x):=\mathbf b_x^\alpha$. While we cannot guarantee that $f_m$ is a $q_m$-elevator, we can nevertheless prove the following subclaim.
\begin{subclaim}\label{798} Suppose that $\bar e:T_\beta\rightarrow T_{\beta_m}$ is a $p$-elevator with $p\in{\mathbb Q}$,
$m<\omega$ and $\beta<\beta_m$.
If $\alpha\notin E$ or if $\{\mathbf b_{\langle1\rangle}^\alpha\restriction \beta_m,\mathbf b_{\langle2\rangle}^\alpha\restriction \beta_m\}\nsubseteq\im(\bar e)$,
then $f_m\circ\bar e$ is a $(p+q_m)$-elevator from $T_\beta$ to $T_\alpha$.
\end{subclaim}
\begin{proof} It follows from Claim~\ref{cfeature} that $f_m(x) \in U(x,q_m)$ for every $x \in T_{\beta_m}$.
Furthermore, by Claim~\ref{cfeature} and Clause~\ref{beth}, for every $(x_0,x_1)\in (T_{\beta_m})^2$,
if $a_\epsilon\nsubseteq^2 (x_0,x_1)$ for every $\epsilon\in E\cap(\beta_m+1)$, then $a_\epsilon\nsubseteq^2 (f_m(x_0),f_m(x_1))$ for every $\epsilon\in E\cap\alpha$.
So if, in addition, $\alpha\notin E$, then $f_m$ is a $q_m$-elevator, and then $f_m\circ\bar e$ is a $(p+q_m)$-elevator by Remark~\ref{elevator-composition}.

On the other hand, if $\alpha\in E$, then the fact that $f_m$ is \emph{not} a $q_m$-elevator is only because of the pair
$(x_0,x_1):=(\mathbf b_{\langle1\rangle}^\alpha\restriction \beta_m,\mathbf b_{\langle2\rangle}^\alpha\restriction \beta_m)$ in $(T_{\beta_m})^2$,
which satisfies $(f_m(x_0),f_m(x_1)) = (\mathbf b_{\langle1\rangle}^\alpha,\mathbf b_{\langle2\rangle}^\alpha) = a_\alpha$,
so that $(x_0,x_1) \not\LE (f_m(x_0),f_m(x_1))$ for this pair.
Thus, we still have that both $f_m\restriction(T_{\beta_m}\setminus\{\mathbf b_{\langle1\rangle}^\alpha\restriction \beta_m\})$
and $f_m\restriction(T_{\beta_m}\setminus\{\mathbf b_{\langle2\rangle}^\alpha\restriction \beta_m\})$
are $q_m$-elevators. So, for every $i\in\{1,2\}$ such that $\mathbf b_{\langle i\rangle}^\alpha\restriction \beta_m\notin\im(\bar e)$,
it is the case that $$f_m\circ\bar e=(f_m\restriction(T_{\beta_m}\setminus\{\mathbf b_{\langle i\rangle}^\alpha\restriction \beta_m\}))\circ\bar e$$
is the composition of a $q_m$-elevator and a $p$-elevator. Again, we are done by Remark~\ref{elevator-composition}.
\end{proof}

Let $\beta<\alpha$ and $q\in\bar{\mathbb Q}$ be given. We go over the clauses of Definition~\ref{Omega_coordinated}:
\begin{enumerate}[start=0]
\item Given a finite $W\s T_\alpha$, first, by possibly enlarging it, we ensure $\mathbf b_{\langle1\rangle}^\alpha\in W$.
We then find a large enough $m<\omega$ such that:
\begin{itemize}
\item $q_m<q$;
\item $\beta_m>\beta$.
\end{itemize}

Set $p:=q-q_m$ and $\bar W:=\{w\restriction \beta_m\mid w\in W\}$. As $T_\beta$ and $T_{\beta_m}$ are coordinated,
we may fix a $p$-elevator $\bar e:T_\beta\rightarrow T_{\beta_m}$ such that $\im(\bar e)\cap \bar W=\emptyset$.
In particular, $\mathbf b_{\langle1\rangle}^\alpha\restriction \beta_m\notin\im(\bar e)$.
By the subclaim, then, $e:=f_m\circ \bar e$ is a $q$-elevator from $T_\beta$ to $T_\alpha$. In addition,
$\im(e)\cap W=\emptyset$, so we are done.

\item Let $x\in T_\beta$, $y\in U(x,q)\cap T_\alpha$ and a finite set $W\s T_\alpha\setminus \{y\}$
be given; we need to find a $q$-elevator $e:T_\beta\rightarrow T_\alpha\setminus W$ such that $e(x)=y$.
First, choose some $i\in\{1,2\}$ such that $\mathbf b_{\langle i\rangle}^\alpha \neq y$.
By possibly enlarging $W$, we may assume that $\mathbf b_{\langle i\rangle}^\alpha\in W$.
Choose a large enough $p\in\mathbb Q\cap(0,q)$ such that $y\in U(x,p)$.
Recalling Claim~\ref{tailf}, we then find a large enough $m<\omega$ such that:
\begin{itemize}
\item $q_m<(q-p)$;
\item $\beta_m>\beta$;
\item $\mathbf b_{w\restriction \beta_m}^\alpha=w$ for every $w\in W\cup\{y\}$.
\end{itemize}

Set $\bar W:=\{w\restriction \beta_m\mid w\in W\}$ and $\bar y:=y\restriction \beta_m$, so that $\bar{y} \in U(x,p) \cap T_{\beta_m}$ and $\bar W\s T_{\beta_m}\setminus\{\bar y\}$.
As $T_\beta$ and $T_{\beta_m}$ are coordinated,
fix a $p$-elevator $\bar e:T_\beta\rightarrow T_{\beta_m}\setminus\bar W$ such that $\bar e(x)=\bar y$.
In particular, $\mathbf b_{\langle i\rangle}^\alpha\restriction \beta_m\notin\im(\bar e)$.
By the subclaim, then, $e:=f_m\circ \bar e$ is a $q$-elevator from $T_\beta$ to $T_\alpha$. In addition,
$\im(e)\cap W=\emptyset$ and $e(x)=y$, so we are done.

\item Let $\{x_0,x_1\}\in [T_\beta]^2$, $(y_0,y_1)\in (U(x_0,q)\cap T_\alpha)\times(U(x_1,q)\cap T_\alpha)$ such that $(x_0,x_1)\LE (y_0,y_1)$,
and a finite set $W\s T_\alpha\setminus \{y_0,y_1\}$ be given;
we need to find a $q$-elevator $e:T_\beta\rightarrow T_\alpha\setminus W$ such that
$e(x_0)=y_0$ and $e(x_1)=y_1$.
Choose a large enough $p\in\mathbb Q\cap(0,q)$ such that $y_0\in U(x_0,p)$ and $y_1\in U(x_1,p)$.
Then, again recalling Claim~\ref{tailf}, find a large enough $m<\omega$ such that:
\begin{itemize}
\item $q_m<(q-p)$;
\item $\beta_m>\beta$;
\item $\mathbf b_{w\restriction \beta_m}^\alpha=w$ for every $w\in W\cup\{y_0,y_1\}$.
\end{itemize}
We now consider three cases:
\begin{enumerate}[label=\textup{(\arabic{enumi}.\arabic*)}]
\item If $\alpha\notin E$, then set $\bar W:=\{w\restriction \beta_m\mid w\in W\}$, $\bar y_0:=y_0\restriction \beta_m$ and $\bar y_1:=y_1\restriction \beta_m$.
Clearly, $(\bar y_0,\bar y_1)\in (U(x_0,p)\cap T_{\beta_m})\times(U(x_1,p)\cap T_{\beta_m})$,
$(x_0,x_1)\LE (\bar y_0,\bar y_1)$ and $\bar W\s T_{\beta_m}\setminus \{\bar y_0,\bar y_1\}$.
As $T_\beta$ and $T_{\beta_m}$ are coordinated, we may fix a $p$-elevator $\bar e:T_\beta\rightarrow T_{\beta_m}\setminus\bar W$ such that $\bar e(x_j)=\bar y_j$ for every $j<2$.
Recalling that $\alpha\notin E$, by the subclaim, $e:=f_m\circ \bar e$ is a $q$-elevator from $T_\beta$ to $T_\alpha$.
In addition, $\im(e)\cap W=\emptyset$ and $e(x_j)=y_j$ for every $j<2$, so we are done.

\item If $\alpha\in E$ but $\{y_0,y_1\}\neq\{\mathbf b_{\langle1\rangle}^\alpha,\mathbf b_{\langle2\rangle}^\alpha\}$,
then choose some $i\in\{1,2\}$ such that $\mathbf b_{\langle i\rangle}^\alpha \notin \{y_0,y_1\}$.
By possibly enlarging $W$, we may assume that $\mathbf b_{\langle i\rangle}^\alpha\in W$.
Set $\bar W:=\{w\restriction \beta_m\mid w\in W\}$, $\bar y_0:=y_0\restriction \beta_m$ and $\bar y_1:=y_1\restriction \beta_m$.
As in the previous case, fix a $p$-elevator $\bar e:T_\beta\rightarrow T_{\beta_m}\setminus\bar W$ such that $\bar e(x_j)=\bar y_j$ for every $j<2$.
In particular, $\mathbf b_{\langle i\rangle}^\alpha\restriction \beta_m\notin\im(\bar e)$.
By the subclaim, then, $e:=f_m\circ \bar e$ is a $q$-elevator from $T_\beta$ to $T_\alpha$.
In addition, $\im(e)\cap W=\emptyset$ and $e(x_j)=y_j$ for every $j<2$, so we are done.

\item If $\alpha\in E$ and for some $j<2$, $$(y_j,y_{1-j})=(\mathbf b_{\langle1\rangle}^\alpha,\mathbf b_{\langle2\rangle}^\alpha)=a_\alpha,$$
then from $(x_0,x_1)\LE (y_0,y_1)$, Clause~\ref{beth} implies that there exists an $\epsilon\in E\cap(\beta+1)$ such that
\begin{align*}
a_\epsilon&\s^2 (x_j,x_{1-j})\\
&\s^2 (y_j,y_{1-j}) =a_\alpha,
\end{align*}
contradicting Claim~\ref{clanti}. So this case does not exist. \qedhere
\end{enumerate}
\end{enumerate}
\end{proof}

Finally, let $T:=\bigcup_{\alpha<\omega_1}T_\alpha$. This completes the construction of our $\aleph_1$-tree.
As $T$ is a subset of $\T$, it is $\mathbb R$-embeddable.

\begin{claim}\label{claim2148} $T$ is almost-Souslin.
\end{claim}
\begin{proof}
Let $A$ be an antichain in $T$,
and we shall show that $H:=\{\dom(a)\mid a\in A\}$ is nonstationary.
For every $i<\omega_1$, let $B_i$ denote the set of all $\beta\in R_i$ such that there exists a countable elementary submodel $\mathcal M\prec H_{\omega_2}$
containing $\{A,T\}$, satisfying $\mathcal M\cap\omega_1=\beta$ and $\mathcal M\cap A=\Omega_\beta$.
\begin{subclaim}\label{applicationoftwofacts} Let $i<\omega_1$.
\begin{enumerate}[label=\textup{(\arabic*)}]
\item $B_i$ is stationary in $\omega_1$;
\item there is a club $D_i\s\omega_1$ such that $\sup(C_\alpha\cap B_i)=\alpha$ for every $\alpha\in D_i$.
\end{enumerate} \end{subclaim}
\begin{proof} (1) Let $C$ be an arbitrary club in $\omega_1$,
and we shall prove that $B_i\cap C\neq\emptyset$.
Set $p:=\{C,A,T\}$ and $\Omega:=A$.
As $\langle \Omega_\beta\mid \beta<\omega_1\rangle$ and $\langle R_i\mid i<\omega_1\rangle$ were given by Fact~\ref{diamond_hw1},
we may find a countable elementary submodel $\mathcal M\prec H_{\omega_2}$ containing $p$ such that $\beta:=\mathcal M\cap\omega_1$ is in $R_i$ and $\mathcal M\cap\Omega=\Omega_\beta$.
As the club $C$ belongs to $\mathcal M$, $\beta\in C$, and hence $\mathcal M$ witnesses that $\beta\in B_i\cap C$.

(2) By Clause~(1) and the fact that $\langle C_\alpha \mid \alpha<\omega_1\rangle$ was given by Fact~\ref{proxystar}.
\end{proof}

Let $\langle D_i\mid i<\omega_1\rangle$ be given by the subclaim. So, the following set is a club in $\omega_1$:
$$D:=\{\alpha\in\Lambda\mid T\restriction\alpha\s \phi[\alpha]\ \&\ \forall i<\alpha\,(\alpha\in D_i)\}.$$

We claim that $D$ is disjoint from $H$. Suppose not; pick $\alpha\in D\cap H$, and then pick $a^*\in A\cap T_\alpha$.
Recalling the definition of $T_\alpha$, we may pick an $x\in T\restriction C_\alpha$ such that $a^*=\mathbf b_x^{\alpha}$.
As $\alpha\in D$, $T\restriction\alpha\s \phi[\alpha]$, so we may fix an $i<\alpha$ such that $\phi(i)=x$.
As $\alpha\in D$ and $i<\alpha$, $\sup(C_\alpha\cap B_i)=\alpha$, so letting $\langle \beta_n\mid n<\omega\rangle$ denote the increasing enumeration of $C_\alpha$,
we may fix a large enough $n<\omega$ such that $\beta_{n+1}\in B_i\setminus(\dom(x)+1)$. In particular, $\beta_{n+1}\in R_i$, so that
$$\psi(\beta_{n+1})=\phi(\pi(\beta_{n+1}))=\phi(i)=x.$$

As $\beta_{n+1}\in B_i$, we may now fix a countable elementary submodel $\mathcal M\prec H_{\omega_2}$
containing $\{A,T\}$ satisfying $\mathcal M\cap\omega_1=\beta_{n+1}$ and $\mathcal M\cap A=\Omega_{\beta_{n+1}}$.
By elementarity of $\mathcal M$, all of the following hold:
\begin{itemize}
\item $\mathcal M\cap T=T\restriction\beta_{n+1}$;
\item $\mathcal M\cap A=A\cap(T\restriction\beta_{n+1})$;
\item for all $y\in T\restriction\beta_{n+1}$ and $q\in{\mathbb Q}$, if $U(y,{q})\cap A$ is nonempty,
then so is $U(y,{q})\cap A\cap(T\restriction\beta_{n+1})$.
\end{itemize}

Consider $y:=b_x^{\alpha}(\beta_n)$. As $\Omega_{\beta_{n+1}}$ is a subset of $T\restriction \beta_{n+1}$,
and $\psi(\beta_{n+1})$ is equal to $x$ which is an element of $T\restriction(C_\alpha\cap \beta_{n+1})$,
we have that $\mathbf b_x^{\alpha}\restriction\beta_{n+1}=e_n(y)$, where $e_n$ is defined either according to Case~I or according to Case~III.
In each case we will derive a contradiction:

$\br$ If $e_n$ is defined according to Case~I, then $\mathbf b_x^{\alpha}\restriction\beta_{n+1}$ belongs to $Q_{n+1}^\alpha$
which means that it extends some element $a$ of $\Omega_{\beta_{n+1}}=A\cap(T\restriction\beta_{n+1})$,
and then $a\s \mathbf b_x^{\alpha}\restriction\beta_{n+1}\s \mathbf b_x^{\alpha}=a^*$, contradicting the fact that $a$ and $a^*$ are two distinct elements of the antichain $A$.

$\br$ If $e_n$ is defined according to Case~III, then Claim~\ref{cfeature} implies that $a^*=\mathbf b_x^{\alpha}\in U(y,\frac{q_n}{4})$. In particular, $U(y,\frac{q_n}{2})\cap A\neq\emptyset$.
By elementarity as above, it follows that $U(y,\frac{q_n}{2})\cap A\cap (T\restriction\beta_{n+1})$ --- equivalently, $U(b_x^{\alpha}(\beta_n),\frac{q_n}{2})\cap \Omega_{\beta_{n+1}}$ --- is nonempty.
Choose $a \in U(b_x^{\alpha}(\beta_n),\frac{q_n}{2})\cap \Omega_{\beta_{n+1}}$.
Choose a large enough $p\in\mathbb{Q}\cap(0,\frac{q_n}{2})$ such that $a \in U(b_x^{\alpha}(\beta_n),p)$,
and set $r := \frac{q_n}{2} -p$.
As $T_{\dom(a)}$ and $T_{\beta_{n+1}}$ are coordinated, fix an $r$-elevator $e: T_{\dom(a)} \to T_{\beta_{n+1}}$, and set $t := e(a)$.
Then $t \in U(b_x^{\alpha}(\beta_n),\frac{q_n}{2})\cap T_{\beta_{n+1}}$ is an extension of $a$.
In particular, $t$ witnesses that $Q_{n+1}^{\alpha}$ is nonempty, contradicting the fact that we are in Case~III.
\end{proof}

Since $T$ is $\mathbb R$-embeddable and almost-Souslin, \cite[Theorem~3]{MR679075} implies that $X_T$ is perfect.
As a topological space is perfectly normal iff it is perfect and normal, our next task is proving that $X_T$ is normal.
Recalling Fact~\ref{devlinshelah}(2), we now turn to prove the following.

\begin{claim}\label{2198} $T$ has property $\gamma$.
\end{claim}
\begin{proof} Let $A\s T$ be a given antichain, and we shall find a club $D\s\omega_1$ and two disjoint open sets $U$ and $V$ such that $A\s U$ and $T\restriction D\s V$.
Consider the club $D$ from the proof of Claim~\ref{claim2148}, which we already know is disjoint from $H:=\{\dom(a)\mid a\in A\}$, that is, $A\cap (T\restriction D)=\emptyset$.

As $A$ is an antichain, it is closed discrete, so by Fact~\ref{devlinshelah}(1), we may fix a pairwise disjoint family $\langle O_a\mid a\in A\rangle$
such that, for every $a\in A$, $O_a$ is an open neighborhood of $a$. For every $a\in A$, we define an open subset $U_a\s O_a$ as follows:
\begin{itemize}
\item if $\dom(a)\notin\Lambda$, then let $U_a:=\{a\}$;
\item otherwise, pick $a'\subsetneq a$ such that $(a',a]\s O_a$, and then, using $\dom(a)\in\Lambda\setminus D$ and the definition of $c(a)$,
pick some $a''\in T$ such that
\begin{itemize}
\item $a' \subsetneq a'' \subsetneq a$,
\item $c(a'')\ge\frac{c(a')+c(a)}{2}$, and
\item $\dom(a'')>\sup(D\cap\dom(a))$,
\end{itemize}
and finally let $U_a:=(a'',a]$.
\end{itemize}

Altogether, $U:=\bigcup_{a\in A}U_a$ is an open neighborhood of $A$.
Our next task is defining, for every $z\in T\restriction D$, an open neighborhood $V_z$ of $z$ which is disjoint from $U$.
This way, $V:=\bigcup_{z\in T\restriction D}V_z$ together with $U$ will be as sought.

To this end, let $z\in T\restriction D$. Write $\alpha:=\dom(z)$ and let $\langle \beta_n\mid n<\omega\rangle$ denote the increasing enumeration of $C_\alpha$.
Pick $x\in T\restriction C_\alpha$ such that $z=\mathbf b_x^\alpha$. As in the proof of Claim~\ref{claim2148}, fix a large enough $n<\omega$ such that all of the following hold:
\begin{itemize}
\item $\beta_{n+1}>\dom(x)$;
\item $\psi(\beta_{n+1})=x$;
\item there exists a countable elementary submodel $\mathcal M\prec H_{\omega_2}$
containing $\{A,T\}$ satisfying $\mathcal M\cap\omega_1=\beta_{n+1}$ and $\mathcal M\cap A=\Omega_{\beta_{n+1}}$.
\end{itemize}

Recall that $z\restriction\beta_{n+1}=e_n(b_x^{\alpha}(\beta_n))$, where $e_n$ is defined either according to Case~I or according to Case~III,
and $z\restriction\beta_n = b_x^\alpha(\beta_n)$.
Before we can define $V_z$, we need to observe the following.
\begin{subclaim}\label{c21881} Suppose that $e_n$ is defined according to Case~III.
For every $a\in A$ such that $U_a\cap (z\restriction \beta_{n+1},z]\neq\emptyset$, both of the following hold:
\begin{enumerate}[label=\textup{(\arabic*)}]
\item $a$ and $z$ are incomparable;
\item $\dom(a)\in\Lambda$ and $a'\subsetneq z\restriction \beta_n$.
\end{enumerate}
\end{subclaim}
\begin{proof} By Claim~\ref{cfeature}, we know that $z\in U(z\restriction{\beta_n},\frac{q_n}{4})$.
Fix a given $a \in A$ as in the hypothesis.

(1) Suppose that $a$ and $z$ are comparable, and we will derive a contradiction in each of the following three cases:
\begin{itemize}[label=$\blacktriangleright$]
\item If $a\subsetneq z\restriction \beta_{n+1}$, then $\dom(a)<\beta_{n+1}$, so that $U_a\cap (z\restriction\beta_{n+1},z]=\emptyset$.
\item If $z\restriction \beta_{n+1}\s a\s z$, then $a\in U(z\restriction{\beta_n},\frac{q_n}{4}) \cap A$,
so that the reflection argument at the end of the proof of Claim~\ref{claim2148} implies that we must have been in Case~I.
\item If $z\subsetneq a$, then $\dom(z)\in D\cap\dom(a)$, so that either $U_a = \{a\}$, or $U_a=(a'',a]$ with
$$\dom(a'')>\sup(D\cap\dom(a))\ge\dom(z),$$
in either case yielding $U_a \cap (z\restriction \beta_{n+1},z] = \emptyset$.
\end{itemize}

\phantomsection\label{c218812}
(2) From $U_a\cap (z\restriction \beta_{n+1},z]\neq\emptyset$ we infer that $\dom(a)\in\Lambda$ and, in addition to $a'$, $z\restriction \beta_{n}$ is below $a$.
So either $a'\subsetneq z\restriction \beta_{n}$ or $z\restriction \beta_{n}\s a'$.

Towards a contradiction, suppose that $z\restriction \beta_{n}\s a'$.
If $a\in U(z\restriction\beta_n,\frac{q_n}{2})$, then again by the argument of the last paragraph of the proof of Claim~\ref{claim2148}, we must have been in Case~I.
Thus, in fact, $c(a) \geq c(z\restriction\beta_n) + \frac{q_n}{2}$.
Together with $a'\supseteq z\restriction \beta_{n}$, this yields
$$c(a'')\ge\frac{c(a')+c(a)}{2}\ge\frac{c(z\restriction \beta_{n})+c(z\restriction \beta_{n})+\frac{q_n}{2}}{2}=c(z\restriction \beta_{n})+\frac{q_n}{4}.$$

Recalling that $z\in U(z\restriction{\beta_n},\frac{q_n}{4})$, we get $$c(a'')\ge c(z\restriction \beta_{n})+\frac{q_n}{4}>c(z),$$
which implies that the intervals $U_a=(a'',a]$ and $(\emptyset,z]$ are disjoint, contradicting the hypothesis.
\end{proof}

Finally, the definition of $V_z$ is divided as follows, where in all cases $V_z$ will be $(z\restriction\beta_m,z]$ for some $m\in[n+1,\omega)$:

$\br$ If $e_n$ is defined according to Case~I, then let $V_z:=(z\restriction \beta_{n+1},z]$.

$\br$ If $e_n$ is defined according to Case~III, then we consider two subcases:

$\br\br$ If $(z\restriction \beta_{n+1},z]\cap U=\emptyset$, then again let $V_z:=(z\restriction \beta_{n+1},z]$.

$\br\br$ Otherwise, pick $a\in A$ such that $(z\restriction \beta_{n+1},z]\cap U_a\neq\emptyset$.
By Subclaim~\ref{c21881}(1), $a$ and $z$ are incomparable,
so we may fix a large enough $m\in[n+1,\omega)$ such that $(z\restriction \beta_m,z]\cap U_a=\emptyset$,
and we let $V_z:=(z\restriction \beta_m,z]$.

\begin{subclaim} $V_z\cap U=\emptyset$.
\end{subclaim}
\begin{proof} Suppose not, and pick $b\in A$ such that $V_z\cap U_b\neq\emptyset$.
Let us come back to the above division, deriving a contradiction in each case:

$\br$ If $e_n$ is defined according to Case~I, then $z\restriction\beta_{n+1} = e_n(b_x^\alpha(\beta_n))$ belongs to $Q_{n+1}^\alpha$,
which means that it extends some element $a$ of $\Omega_{\beta_{n+1}}=A\cap(T\restriction\beta_{n+1})$.
But since $(z\restriction \beta_{n+1},z]\cap (b'',b]\neq\emptyset$, we get that $z\restriction \beta_{n+1}\subsetneq b$,
so that $a$ and $b$ are two distinct comparable elements of the antichain $A$. This is a contradiction.

$\br$ If $e_n$ is defined according to Case~III, then since $(z\restriction \beta_{n+1},z]\cap U\neq\emptyset$,
it was the case that we picked an $a\in A$ such that $(z\restriction \beta_{n+1},z]\cap U_a\neq\emptyset$,
and let $V_z:=(z\restriction \beta_m,z]$ for some $m\in[n+1,\omega)$ such that $(z\restriction \beta_m,z]\cap U_a=\emptyset$.
In particular, $V_z\cap U_a=\emptyset$, so $a\neq b$.
By Subclaim~\ref{c21881}(2), $a'\subsetneq z\restriction \beta_n$ and $b'\subsetneq z\restriction \beta_n$,
so $z\restriction \beta_n\in(a',a]\cap(b',b]\s O_a\cap O_b$, contradicting the fact that $O_a$ and $O_b$ are disjoint.
\end{proof}
This completes the proof.
\end{proof}

Next, to help us verify that $(X_T)^2$ is not cmc, we define an auxiliary map $g:E\rightarrow\omega$ via
$g(\alpha):=\min(K_\alpha)$.

\begin{claim}\label{claim285} For every function $f:T^2\rightarrow\omega$ there are stationarily many $\alpha\in E$
such that
$$\sup\{\tau<\alpha\mid f(\mathbf b_{\langle1\rangle}^\alpha\restriction\tau,\mathbf b_{\langle2\rangle}^\alpha\restriction\tau)=g(\alpha)\}=\alpha.$$
\end{claim}
\begin{proof} Given $f:T^2\rightarrow \omega$, for every $k<\omega$, let $B_k$ be the set of all $\beta\in R_{\phi^{-1}(k)}$ such that there exists a countable elementary submodel $\mathcal M\prec H_{\omega_2}$ such that all of the following hold:
\begin{itemize}
\item $\{T, f\}\in \mathcal M$;
\item $\beta=\mathcal M\cap\omega_1$, and
\item $\Omega_\beta=f\restriction(T\restriction\beta)^2$.
\end{itemize}

A proof similar to that of Subclaim~\ref{applicationoftwofacts} establishes that
$$A:=\{\alpha\in B_0\mid \forall k<\omega\,[\sup(C_\alpha\cap B_k)=\alpha]\}$$ is a stationary subset of $B_0$.

\begin{subclaim}\label{ase} $A\s E$.
\end{subclaim}
\begin{proof} Suppose not, and let $\alpha\in A\setminus E$.
Let $\langle \beta_n\mid n<\omega\rangle$ denote the increasing enumeration of $C_\alpha$.
By Claim~\ref{cfeature}, for every $n<\omega$,
$$(b_{\langle1\rangle}^\alpha(\beta_n),b_{\langle2\rangle}^\alpha(\beta_n))\LE (b_{\langle1\rangle}^\alpha(\beta_{n+1}),b_{\langle2\rangle}^\alpha(\beta_{n+1})).$$
So, since $\alpha\notin E$, we obtain furthermore that, for every $n<\omega$,
$$(b_{\langle1\rangle}^\alpha(\beta_n),b_{\langle2\rangle}^\alpha(\beta_n))\LE (\mathbf b_{\langle1\rangle}^\alpha,\mathbf b_{\langle2\rangle}^\alpha).$$

As $\alpha\in A\s B_0$, the following equation holds
$$K_\alpha=\{k<\omega\mid \sup\{\tau<\alpha\mid f(\mathbf b_{\langle1\rangle}^\alpha\restriction\tau,\mathbf b_{\langle2\rangle}^\alpha\restriction\tau)=k\}=\alpha\}.$$
We shall show that $k:=f(\mathbf b_{\langle1\rangle}^\alpha,\mathbf b_{\langle2\rangle}^\alpha)$ belongs to $K_\alpha$,
in particular, $K_\alpha\neq\emptyset$, contradicting the fact that $\alpha\notin E$.

From $\alpha \in A$, we infer that the set $N:=\{ n<\omega\mid \beta_{n+1}\in B_k\}$ is infinite.
Thus, to prove that $k\in K_\alpha$, it suffices to prove that for every $n\in N$,
there exists some $\tau\in[\beta_n,\beta_{n+1})$ such that $f(\mathbf b_{\langle1\rangle}^{\alpha}\restriction\tau,\mathbf b_{\langle2\rangle}^{\alpha}\restriction\tau)=k$.

To this end, let $n\in N$.
As $\Omega_{\beta_{n+1}}=f\restriction (T\restriction\beta_{n+1})^2$ and $\psi(\beta_{n+1})=k$,
we have that $b_{\langle j\rangle}^{\alpha}(\beta_{n+1})=e_n(b_{\langle j\rangle}^{\alpha}(\beta_{n}))$ for $j\in\{1,2\}$,
where $e_n$ is defined either according to Case~II or according to Case~III.

$\br$ If $e_n$ is defined according to Case~II, then $(b_{\langle1\rangle}^{\alpha}(\beta_{n+1}),b_{\langle2\rangle}^{\alpha}(\beta_{n+1}))$ belongs to $P_{n+1}^\alpha$
which means that there indeed exists some $\tau\in[\beta_n,\beta_{n+1})$ such that $f(\mathbf b_{\langle1\rangle}^{\alpha}\restriction\tau,\mathbf b_{\langle2\rangle}^{\alpha}\restriction\tau)=k$.

$\br$ If $e_n$ is defined according to Case~III, then Claim~\ref{cfeature} implies that $\mathbf b_{\langle j\rangle}^{\alpha}\in U(b^\alpha_{\langle j\rangle}(\beta_n),\frac{q_n}{4})\s U(b^\alpha_{\langle j\rangle}(\beta_n),\frac{q_n}{2})$ for $j\in\{1,2\}$.
Pick a countable $\mathcal M\prec H_{\omega_2}$ that includes $\{T,f\}$ such that $\mathcal M\cap\omega_1=\beta_{n+1}$ and $\Omega_{\beta_{n+1}}=f\restriction(T\restriction\beta_{n+1})^2$.
So, $\mathcal M$ reflects the above properties of the pair $(\mathbf b_{\langle1\rangle}^\alpha,\mathbf b_{\langle2\rangle}^\alpha)$, meaning that there exists a pair $(r_1,r_2)\in (T\restriction\beta_{n+1})^2$ such that
\begin{itemize}
\item $r_1\in U(b_{\langle1\rangle}^\alpha(\beta_n),\frac{q_n}{2})$,
\item $r_2\in U(b_{\langle2\rangle}^\alpha(\beta_n),\frac{q_n}{2})$,
\item $(b_{\langle1\rangle}^\alpha(\beta_n),b_{\langle2\rangle}^\alpha(\beta_n))\LE (r_1,r_2)$, and
\item $f(r_1,r_2)=k$.
\end{itemize}

As $T_{\dom(r_1)}$ and $T_{\beta_{n+1}}$ are coordinated, it follows that $P_{n+1}^\alpha$ is nonempty, contradicting the fact that we are in Case~III.
\end{proof}

Consider any given $\alpha\in A$, so that $\Omega_\alpha=f\restriction(T\restriction\alpha)^2$.
By Subclaim~\ref{ase}, $\alpha\in E$, which means that $K_\alpha\neq \emptyset$. So, by the definition of $g$,
$$\sup\{\tau<\alpha\mid f(\mathbf b_{\langle1\rangle}^\alpha\restriction \tau,\mathbf b_{\langle2\rangle}^\alpha\restriction\tau)=g(\alpha)\}=\alpha,$$
as sought.
\end{proof}

\begin{claim}\label{21910} $(X_T)^2$ is not cmc.
\end{claim}
\begin{proof} By \cite[Theorem~2.2]{MR353254}, a space is cmc iff for every $\s$-decreasing sequence $\langle D_n\mid n<\omega\rangle$ of closed sets that vanishes (i.e., $\bigcap_{n<\omega}D_n=\emptyset$),
there exists a $\s$-decreasing vanishing sequence $\langle U_n\mid n<\omega\rangle$ of open sets such that $U_n\supseteq D_n$ for every $n<\omega$.

As $\{ a_\epsilon\mid\epsilon\in E\}$ is a $\s^2$-antichain in $T^2$, Proposition~\ref{prop26} implies that
for every $n<\omega$, $D_n:=\{ a_\epsilon\mid \epsilon\in E\ \&\allowbreak\ g(\epsilon)\ge n\}$ is a closed (discrete) subset of $(X_T)^2$.
Clearly, $\langle D_n\mid n<\omega\rangle$ is decreasing and vanishing.
Towards a contradiction, suppose that $\langle U_n\mid n<\omega\rangle$ is a vanishing decreasing sequence of open subsets of $(X_T)^2$ such that $U_n\supseteq D_n$ for every $n<\omega$.
Define a function $f:T^2\rightarrow\omega$ by letting $f(x_0,x_1)$ be the least $n<\omega$ such that $(x_0,x_1)\notin U_n$. Now using Claim~\ref{claim285} pick $\alpha\in E$ such that
$$\sup\{\tau<\alpha\mid f(\mathbf b_{\langle1\rangle}^\alpha\restriction\tau,\mathbf b_{\langle2\rangle}^\alpha\restriction\tau)\le g(\alpha)\}=\alpha.$$

Consider $n:=g(\alpha)$, so that $(\mathbf b_{\langle1\rangle}^\alpha,\mathbf b_{\langle2\rangle}^\alpha)=a_\alpha\in D_n\s U_n$. Since $U_n$ is open in $(X_T)^2$,
we may fix $y_1\subsetneq\mathbf b_{\langle1\rangle}^\alpha$ and $y_2\subsetneq\mathbf b_{\langle2\rangle}^\alpha$ such that $(y_1,\mathbf b_{\langle1\rangle}^\alpha]\times(y_2,\mathbf b_{\langle2\rangle}^\alpha]\s U_n$.
Fix a large enough $\tau<\alpha$ such that $\max\{\dom(y_1),\dom(y_2)\}<\tau$ and
$f(\mathbf b_{\langle1\rangle}^\alpha\restriction\tau,\mathbf b_{\langle2\rangle}^\alpha\restriction\tau)\le n$.

It follows that $(\mathbf b_{\langle1\rangle}^\alpha\restriction\tau,\mathbf b_{\langle2\rangle}^\alpha\restriction\tau)\in(y_1,\mathbf b_{\langle1\rangle}^\alpha]\times(y_2,\mathbf b_{\langle2\rangle}^\alpha]\s T^2 \cap U_n$,
which must mean that $f(\mathbf b_{\langle1\rangle}^\alpha\restriction\tau,\mathbf b_{\langle2\rangle}^\alpha\restriction\tau)>n$.
This is a contradiction.
\end{proof}
This completes the proof.
\end{proof}
\begin{remark}\label{1181}
\begin{enumerate}
\item By Claim~\ref{2198} and \cite[Theorem~2.1]{MR686092}, $X_T$ is hereditarily collectionwise normal.
\item In \cite{MR1261168}, Be\v{s}lagi\'{c} constructed from $\diamondsuit(\omega_1)$ a perfectly normal space $X$ such that $X^2$ is not cmc
but is normal. In contrast, our tree-based example witnesses the non-productivity of normality.
This is because the proof of Claim~\ref{21910} moreover shows that the closed subspace $T^2$ of $(X_T)^2$ is not cmc.
So, by \cite[Corollary~4.2]{MR577767}, the closed subspace $T^2$ cannot be normal, let alone $(X_T)^2$.
\item An inspection of the preceding construction makes it clear that for the sake of getting an $\mathbb R$-embeddable $\aleph_1$-tree $T$ such that $T^2$ is not cmc,
$\diamondsuit(\omega_1)$ suffices. As $(T^2,\s^2)$ would be an $\mathbb R$-embeddable $\aleph_1$-tree in this case, this shows that $\diamondsuit(\omega_1)$ yields an $\mathbb R$-embeddable $\aleph_1$-tree that is not cmc,
a result announced by Hanazawa on \cite[p.~61]{MR679075}.\footnote{The first consistent example of an $\mathbb R$-embeddable $\aleph_1$-tree that is not special was given by Baumgartner in his dissertation \cite{MR2620340}.}
\item A minor tweak of the preceding construction will secure the following anti-uniformization feature, a strong form of Claim~\ref{claim285}.
There exist an $\omega$-bounded ladder system $\langle A_\alpha\mid \alpha<\omega_1\rangle$ and a function $g:T^2\rightarrow\omega$ such that
for every function $f:T^2\rightarrow\omega$ there are stationarily many $\alpha\in E$ such that for every $(z_0,z_1)\in T_\alpha\times T_\alpha$,
$$\sup\{\tau\in A_\alpha\mid f(z_0\restriction\tau,z_1\restriction\tau)\le g(z_0,z_1)\}=\alpha.$$
Such an almost-Souslin tree can be viewed as an instance of diamond whose square is an instance of uniformization.
\item The preceding construction can also be tweaked to obtain, for every positive integer $k$, an $\aleph_1$-tree $T\s{}^{<\omega_1}\omega$
consisting of a finite-to-one maps (in particular, $T$ is $\mathbb R$-embeddable)
such that $(T^k,{\s^k})$ is almost-Souslin and $T^{k+1}$ is not cmc.
\item It is an open problem whether every $\mathbb R$-embeddable cmc $\aleph_1$-tree is perfect.
An affirmative answer holds assuming $\ma$ \cite[Corollary~23B/41J]{fremlin1984consequences} or $\VisL$ or $\pmea$ \cite[p.~420]{MR1244382},
but the general case is open. A purported counterexample from $\diamondsuit^*(\omega_1)$ was constructed in \cite[Theorem~4]{MR679075},
but it is flawed in view of the fact that $\mathsf V=\mathsf L$ implies $\diamondsuit^*(\omega_1)$.
The error was acknowledged in \cite[p.~20]{MR769446}.
\end{enumerate}
\end{remark}

\subsection{Additional set-theoretic comments}\label{subsection:additional-set-theory}
The topologist who has had enough set theory can stop reading here.
For the interested reader, we mention that the ladder system in the conclusion of Fact~\ref{proxystar} is a particular instance of the proxy principles $\p^-(\ldots)$ from \cite{paper65},
namely, $\p^-_\omega(\omega_1,2,{\sq},1,\allowbreak\ns^+_{\omega_1},2,1)$.
The existence of such a ladder system does not require $\ch$, let alone $\diamondsuit^*(\omega_1)$,
as it is for instance introduced by the forcing to add a single Cohen real (see \cite[Claim~2.3.3]{rinot12}).

Even the conjunction $\p^-_\omega(\omega_1,2,{\sq},1,\ns^+_{\omega_1},2,1) \land \diamondsuit(\omega_1)$,
which is denoted by $\p_\omega(\omega_1,2,{\sq},1,\ns^+_{\omega_1},2,1)$ according to \cite[Definition~1.6]{paper22},
is strictly weaker than $\diamondsuit^*(\omega_1)$, as can be seen from the proof of Corollary~\ref{cor21} below.
Meanwhile, since the application of $\diamondsuit^*(\omega_1)$ in the proof of Theorem~\ref{thm211} is limited to Facts \ref{proxystar} and \ref{diamond_hw1},
we arrive at the following conclusion.
\begin{thm}\label{thmA-from-proxy}
$\p_\omega(\omega_1,2,{\sq},1,\ns^+_{\omega_1},2,1)$ implies the existence of
an $\mathbb R$-embeddable almost-Souslin $\aleph_1$-tree $T$ such that $X_T$ is perfectly normal,
but $(X_T)^2$ is not cmc.
\qed
\end{thm}

\begin{cor}\label{cor21} Suppose that $\diamondsuit(\omega_1)$ holds.
If $\mathbb P$ is the notion of forcing to add a single Cohen real
or any other ccc notion of forcing of size at most $\aleph_1$ that is not ${}^\omega\omega$-bounding,
then in $V^{\mathbb P}$ there is an $\mathbb R$-embeddable almost-Souslin $\aleph_1$-tree $T$ such that $X_T$ is perfectly normal, but $(X_T)^2$ is not cmc.
\end{cor}
\begin{proof} Given $\mathbb P$ as in the hypothesis, since we have assumed $\diamondsuit(\omega_1)$ in $V$,
by \cite[Proposition~3.1(1) and Theorem~3.4]{paper26}, both $\p^*(E^{\omega_1}_\omega,\omega)$ and $\diamondsuit(\omega_1)$ hold in $V^{\mathbb P}$.\footnote{See also the third section of~\cite[Table~3.2]{paper65}.}
By Remark~(iv) after \cite[Definition~3.3]{paper26}, $\p^*(E^{\omega_1}_\omega,\omega)$ stands for $\p^-_\omega(\omega_1,\infty,{\sq},1,\ns^+_{\omega_1},2,{<}\infty)$,
which is easily seen to imply $\p^-_\omega(\omega_1,\allowbreak 2,{\sq},1,\allowbreak\ns^+_{\omega_1},2,1)$.
Altogether, $\p_\omega(\omega_1,2,{\sq},1,\allowbreak\ns^+_{\omega_1},2,1)$ holds in $V^{\mathbb P}$,
and we can appeal to Theorem~\ref{thmA-from-proxy}.\footnote{To demonstrate that the hypothesis of Theorem~\ref{thmA-from-proxy}
is consistently weaker than that of Theorem~\ref{thm211},
note that if $\diamondsuit^*(\omega_1)$ failed in $V$, then it remains failing in $V^{\mathbb P}$. Indeed, by \cite[Lemma~VII.5.5 and Exercise~VII.H1]{MR597342},
a ccc notion of forcing cannot force $\diamondsuit^*(\omega_1)$.}
\end{proof}

An argument of Todor{\v{c}}evi{\'c} \cite[p.~44]{MR2213652} building on \cite[Theorem~1.3]{MR535835} shows that an extra hypothesis concerning the ground model is indeed necessary here.
For instance, it shows that in the forcing extension by adding a single Cohen or random real over a model of $\ma$,
for every $\mathbb R$-embeddable $\aleph_1$-tree $T$, the space $(X_T)^2$ is cmc (in fact, it is a $Q$-space).

\section{Clause~(1) of Theorem~\ref{thmb}}\label{section:special}
Throughout this section, $\kappa$ stands for a regular uncountable cardinal,
$\reg(\kappa)$ denotes the collection of all infinite regular cardinals below $\kappa$,
and for $\sigma\in\reg(\kappa)$, we let $E^\kappa_\sigma:=\{\alpha < \kappa \mid \cf(\alpha) = \sigma\}$.
For a set of ordinals $A$ we write $\acc(A) := \{\alpha \in A \mid \sup(A \cap \alpha) = \alpha > 0\}$ and $\nacc(A) := A \setminus \acc(A)$.
\begin{defn} A \emph{$C$-sequence over $\kappa$} is a sequence $\vec C=\langle C_\alpha\mid\alpha<\kappa\rangle$
such that, for every $\alpha<\kappa$, $C_\alpha$ is a closed subset of $\alpha$ with $\sup(C_\alpha)=\sup(\alpha)$.
\end{defn}

\begin{defn}[{\cite[\S4]{paper71}}] Let $\vec C=\langle C_\alpha\mid\alpha<\kappa\rangle$ be a $C$-sequence.
\begin{itemize}
\item The set of \emph{lower-regressive levels} of $\vec{C}$ is the following:
$$R(\vec{C}):=\{\alpha\in\acc(\kappa)\mid \forall\beta\in(\alpha,\kappa)\,(\otp(C_\beta\cap\alpha)<\alpha)\}.$$
\item The set of \emph{avoiding levels} of $\vec{C}$ is the following:
$$A(\vec{C}):=\{\alpha\in\acc(\kappa)\mid \forall\beta\in(\alpha,\kappa)\,(\sup(C_\beta\cap\alpha)<\alpha)\}.$$
\end{itemize}
\end{defn}

As explained in \cite{paper65}, constructions of $\kappa$-trees are typically guided by some `good' $C$-sequences
such as those whose existence is asserted by \emph{proxy principles}. The general definition of these principles may be found as \cite[Definition~2.5, \S2.4 and \S2.6]{paper65},
but for our purpose, the following special case suffices.\footnote{Compare the upcoming definition with the conclusion of Fact~\ref{proxystar}.}

\begin{defn}\label{def32}
$\p_{<}(\kappa,2,{\sq},\kappa)$ asserts $\diamondsuit(\kappa)$ as well as the existence of a $C$-sequence $\vec C=\langle C_\alpha\mid\alpha<\kappa\rangle$ such that all of the following hold:
\begin{enumerate}[label=\textup{(\arabic*)}]
\item\label{321} $R(\vec{C})$ covers a club;
\item\label{322} for all $\alpha<\kappa$ and all $\beta\in\acc(C_\alpha)$, $C_\beta=C_\alpha\cap\beta$;
\item\label{323} for every sequence $\langle B_i\mid i<\kappa\rangle$ of cofinal subsets of $\kappa$,
there are stationarily many $\alpha<\kappa$ such that, for every $i<\alpha$,
$$\sup(\nacc(C_\alpha)\cap B_i)=\alpha.$$
\end{enumerate}
\end{defn}
\begin{remark}\label{rmk33} To compare, the proxy principle $\p_\xi(\kappa,2,{\sq},\kappa)$ is the outcome of replacing Clause~\ref{321} by the uniform requirement that $\otp(C_\alpha)\le\xi$ for all $\alpha<\kappa$.
For a successor cardinal $\kappa$, having $\p_\xi(\kappa,2,{\sq},\kappa)$ with $\xi<\kappa$ is consistent \cite[Theorem~3.6]{paper22} and quite useful \cite[Main Theorem]{MR4833803},
but what do we do at inaccessibles?
A natural generalization would require that $\otp(C_\alpha)<\alpha$ for club many $\alpha<\kappa$. However, for an inaccessible cardinal $\kappa$, every $\vec C$ satisfying \ref{322} and \ref{323}
must satisfy that $\{\alpha\in E^\kappa_\omega\mid \otp(C_\alpha)=\alpha\}$ is stationary in $\kappa$. Thus, Clause~\ref{321} turns out to be the right way to go.
\end{remark}

We now introduce a poset $\mathbb P$ that forces $\p_{<}(\kappa,2,{\sq},\kappa)$. In the next section we shall demonstrate the utility of this axiom.

\begin{defn}\label{def01}
Let $\mathbb P$ be the forcing notion in which each condition is either $\emptyset$ or a pair $p = (\langle C^p_\alpha \mid \alpha \le \gamma^p \rangle,D^p)$ such that:
\begin{enumerate}
\item $\gamma^p\in\acc(\kappa)$;
\item $D^p$ is a closed set of limit ordinals with $\max(D^p)=\gamma^p$;
\item for all $\alpha \le \gamma^p$, $C^p_\alpha$ is a closed subset of $\alpha$ with $\sup(C_\alpha^p)=\sup(\alpha)$;
\item for all $\alpha < \beta \le \gamma^p$:
\begin{enumerate}
\item if $\alpha \in \acc(C^p_\beta)$, then $C^p_\beta \cap \alpha = C^p_\alpha$;
\item if $\alpha\in D^p$, then $\otp(C^p_\beta\cap\alpha)<\alpha$.
\end{enumerate}
\end{enumerate}

We order $\mathbb P$ by assigning $\emptyset$ as the maximal element $\one_{\mathbb P}$, and otherwise requiring end-extension on both coordinates.
\end{defn}

Before we analyze the features of the poset $\mathbb P$, we recall the notion of strategic closure.

\begin{defn}\label{thegame} For a notion of forcing $\mathbb P$ and an ordinal $\sigma$,
$\Game_\sigma(\mathbb P)$ denotes the following two-player game of perfect information:

Two players, named {\bf I} and {\bf II}, take turns to play conditions from $\mathbb P$ for $\sigma$ many moves, with
{\bf I} playing at odd stages and {\bf II} at even stages (including all limit stages).
{\bf II} must play $\one_{\mathbb P}$ at move zero.
Let $p_i$ be the condition played at move $i$;
the player who plays $p_i$ loses immediately unless $p_i\le p_j$ for all $j<i$.
If neither player loses at any stage $i<\sigma$, then {\bf II} wins.

$\mathbb P$ is said to be \emph{$\sigma$-strategically-closed} iff {\bf II} has a winning strategy for $\Game_\sigma(\mathbb P)$.
It is said to be \emph{${<}\sigma$-strategically-closed} iff it is $\tau$-strategically-closed for all $\tau<\sigma$.
\end{defn}

\begin{lemma}\label{lemma02}
\begin{enumerate}[label=\textup{(\arabic*)}]
\item\label{c1} for every nonempty $p\in\mathbb P$, for every $\beta<\kappa$, there exists a $q\le p$ with $\gamma^q>\beta$ and $\otp(C^q_{\gamma^q})=\omega$.
\item\label{c2} for every $\mathbb P$-name $\dot B$ for a cofinal subset of $\kappa$,
for every $\sigma\in\acc(\kappa)$,
for every nonempty $p\in\mathbb P$ such that $\otp(C^p_{\gamma^p})<\gamma^p$,
there exists a $q\le p$ such that all of the following hold:
\begin{itemize}
\item $C^p_{\gamma^p}\sq C^{q}_{\gamma^{q}}$;
\item $\otp(C^{q}_{\gamma^{q}})=\otp(C^p_{\gamma^p})+\omega\cdot\sigma<\gamma^{q}$;
\item $q$ forces that $\{ C^q_{\gamma^q}(\otp(C^p_{\gamma^p})+\omega\cdot \iota+1)\mid \iota<\sigma\}$ is a subset of $\dot B$.
\end{itemize}
\item\label{c3} $\mathbb P$ is ${<}\kappa$-strategically-closed.
\item\label{c4} $\mathbb P$ is countably closed.
\item\label{c5} $\mathbb P$ has size $\kappa^{<\kappa}$.
\end{enumerate}
\end{lemma}
\begin{proof} (1) Given a nonempty $p\in\mathbb P$ and $\beta<\kappa$, define $q\le p$ as follows:
\begin{itemize}
\item $\gamma^{q}:=\gamma^{p}+\beta+\omega$;
\item $C_\alpha^{q}:=C_\alpha^{p}$ for every $\alpha\le\gamma^{p}$;
\item $C_\alpha^{q}:=\alpha\setminus\gamma^{p}$ whenever $\gamma^{p}<\alpha<\gamma^{q}$;
\item $C_{\gamma^q}^{q} := \gamma^q\setminus(\gamma^p+\beta)$, a tail of $\gamma^q$ of order-type $\omega$;
\item $D^{q}:=D^{p}\cup \{\gamma^{q}\}$.
\end{itemize}
It is clear that $q$ is as sought.

Next, we prove Clauses \ref{c2} and \ref{c3} simultaneously.
So, let $\dot B$ be a $\mathbb P$-name for a cofinal subset of $\kappa$, let $\sigma\in\acc(\kappa)$,
and we shall play the game $\Game_{\sigma+1}(\mathbb P)$, producing a decreasing sequence of conditions $\langle p_i\mid i\le\sigma\rangle$,
where $p_0=\one_{\mathbb P}$ and $p_1$ is an arbitrary condition (as dictated by the rules of the game),
and $p_2$ is any extension of $p_1$ satisfying $\otp(C^{p_2}_{\gamma^{p_2}})<\gamma^{p_2}$ (such an extension exists, by Clause~\ref{c1}),
so $p_2$ plays the role of $p$ from Clause~\ref{c2}.

While describing the winning strategy for {\bf II} in $\Game_{\sigma+1}(\mathbb P)$,
we shall be producing an auxiliary sequence $\langle \beta_i\mid \text{even }i\in[4,\sigma)\rangle$
of ordinals, and the strategy for {\bf II} will ensure the following features:
\begin{enumerate}[label=(\roman*)]
\item\label{gamei} $\otp(C^{p_2}_{\gamma^{p_2}})<\gamma^{p_2}$;
\item\label{gameii} $\gamma^{p_4}>\gamma^{p_2}+\omega\cdot\sigma$;
\item\label{gameiii} for every even $j\in[4,\sigma]$,
$$\acc(C_{\gamma^{p_j}}^{p_j})=\acc(C_{\gamma^{p_2}}^{p_2})\cup\{ \gamma^{p_i}\mid \text{even }i\in[2,j)\};$$

\item\label{gameiv} for every even $i\in[2,\sigma)$, the ordinal $\beta_{i+2}$ will be equal to $\min(C_{\gamma^{p_{i+2}}}^{p_{i+2}}\setminus (\gamma^{p_{i}}+1))$
and it will be the case that $p_{i+2}\forces \beta_{i+2}\in \dot B$.
\end{enumerate}

Note that the combination of Clauses \ref{gamei}--\ref{gameiii} implies the following:
\begin{enumerate}[label=(\roman*),resume]
\item\label{gamev} for every even $i\in[2,\sigma]$, $\otp(C_{\gamma^{p_i}}^{p_i})<\gamma^{p_i}$,
\end{enumerate}
because for every $j\in[4,\sigma]$, $\otp(C_{\gamma^{p_j}}^{p_j})\le \gamma^{p_2}+\omega\cdot j<\gamma^{p_4}\leq \gamma^{p_j}$.
Additionally, the combination of Clauses \ref{gameiii} and \ref{gameiv} implies the following:
\begin{enumerate}[label=(\roman*),resume]
\item\label{gamevi} for every even $j\in[4,\sigma)$, $\otp(C_{\gamma^{p_j}}^{p_{j}}\setminus\beta_j)=\omega$.
\end{enumerate}

We now turn to the description of the strategy for {\bf II}:
\begin{itemize}[label=$\blacktriangleright$]
\item Start by letting $p_0:=\one_{\mathbb P}$.
\item Once $p_1$ was played, $p_2$ will be an extension of $p_1$ satisfying Clause~\ref{gamei};
\item For every even $i\in[2,\sigma)$ such that $p_{i+1}$ has already been defined, we first let $q_{i+1}$ be some extension of $p_{i+1}$ that decides that some ordinal $\beta_{i+2}$ belongs to $\dot B\setminus(\gamma^{p_{i+1}}+\omega\cdot\sigma+1)$.
By Clause~\ref{c1}, we may assume that $\gamma^{q_{i+1}}>\beta_{i+2}$.
We then define $p_{i+2}$ as follows:

\begin{itemize}[label=\textbullet]
\item $\gamma^{p_{i+2}}:=\gamma^{q_{i+1}}+\omega$;
\item $C_\alpha^{p_{i+2}}:=C_\alpha^{q_{i+1}}$ for every $\alpha\le\gamma^{q_{i+1}}$;
\item $C_{\alpha+1}^{p_{i+2}}:=\{\alpha\}$ for every $\alpha\in\gamma^{p_{i+2}}\setminus \gamma^{q_{i+1}}$;
\item $C^{p_{i+2}}_{\gamma^{p_{i+2}}}:=C^{p_{i}}_{\gamma^{p_{i}}}\cup\{\gamma^{p_i},\beta_{i+2}\}\cup\{\gamma^{q_{i+1}}+n\mid n<\omega\}$;
\item $D^{p_{i+2}}:=D^{q_{i+1}}\cup \{\gamma^{p_{i+2}}\}$.
\end{itemize}

Evidently, $\acc(C^{p_{i+2}}_{\gamma^{p_{i+2}}})=\acc(C^{p_{i}}_{\gamma^{p_{i}}})\cup\{\gamma^{p_i}\}$, so that requirement~\ref{gameiii} is maintained.
Clause~\ref{gameiv} is satisfied by design, and so is Clause~\ref{gameii} in case we are at stage $i=2$.
To verify that $p_{i+2}$ is a legitimate condition, it suffices to focus on Clause~(4)(b) of Definition~\ref{def01},
with $\beta:=\gamma^{p_{i+2}}$ and some $\alpha\in D^{p_{i+2}}\cap\beta$.
That is, to verify that $\otp(C^{p_{i+2}}_{\gamma^{p_{i+2}}}\cap\alpha)<\alpha$ for every $\alpha\in D^{q_{i+1}}$.
As $\max(D^{q_{i+1}})=\gamma^{q_{i+1}}$ and $C^{p_{i+2}}_{\gamma^{p_{i+2}}}\cap\gamma^{q_{i+1}}=C^{p_{i}}_{\gamma^{p_{i}}}\cup\{\gamma^{p_i},\beta_{i+2}\}$,
it suffices to verify that $\otp(C^{p_{i}}_{\gamma^{p_{i}}}\cap\alpha)<\alpha$ for every $\alpha\in D^{q_{i+1}}\cap(\gamma^{p_i}+1)=D^{p_i}$.
For $\alpha<\gamma^{p_i}$, this follows from the fact that $p_i$ is a legitimate condition.
For $\alpha=\gamma^{p_i}$, this follows from Clause~\ref{gamev}.

\item Given a $j\in\acc(\sigma+1)$ such that $\langle p_i\mid i<j\rangle$ has already been determined, define $p_j$ as follows:
\begin{itemize}[label=\textbullet]
\item $\gamma^{p_j}:=\sup_{i<j}\gamma^{p_i}$;
\item $C_\alpha^{p_j}:=C_\alpha^{p_i}$ for every $\alpha<\gamma^{p_j}$, using a large enough $i<j$;
\item $C^{p_j}_{\gamma^{p_j}}:=\bigcup\{ C^{p_{i}}_{\gamma^{p_{i}}}\mid \text{even }i\in[2,j)\}$;
\item $D^{p_j}:=\bigcup_{i<j}D^{p_{i}}\cup \{\gamma^{p_j}\}$.
\end{itemize}

It is clear that requirement~\ref{gameiii} is maintained.
To verify that $p_j$ is a legitimate condition, we verify Clause~(4)(b) of Definition~\ref{def01} with respect to $\beta:=\gamma^{p_j}$ and a given $\alpha\in D^{p_j}\cap\beta$.
Find an even $i<j$ such that $\alpha<\gamma^{p_i}$. Then $C^{p_j}_{\gamma^{p_j}}\cap\alpha=C^{p_i}_{\gamma^{p_i}}\cap\alpha$ and the latter has order-type less than $\alpha$.
Altogether, $p^j$ is a legitimate condition extending $p^i$ for all $i<j$.
\end{itemize}

This completes the presentation of a successful strategy for {\bf II} in the game $\Game_{\sigma+1}(\mathbb P)$.
In addition, $q:=p_\sigma$ satisfies that $q\le p_2$ and all of the following hold:
\begin{itemize}
\item $C^{p_2}_{\gamma^{p_2}}\sq C^{q}_{\gamma^{q}}$;
\item $\otp(C^{q}_{\gamma^{q}})=\otp(C^{p_2}_{\gamma^{p_2}})+\omega\cdot\sigma<\gamma^{q}$;
\item $\{ C^q_{\gamma^q}(\otp(C^{p_2}_{\gamma^{p_2}})+\omega\cdot \iota+1)\mid \iota<\sigma\}$ is equal to $\{\beta_i\mid \text{even }i\in[4,\sigma)\}$
and is forced by $q$ to be subset of $\dot B$.
\end{itemize}

Finally, Clauses \ref{c4} and \ref{c5} are standard and are left to the reader.
\end{proof}

\begin{cor} If $\kappa^{<\kappa}=\kappa$, then $\mathbb P$ preserves all cardinals.
\end{cor}
\begin{proof} By Clauses \ref{c3}--\ref{c5} of Lemma~\ref{lemma02}.
\end{proof}

\begin{cor}\label{cor14} For every $p\in\mathbb P$, for every $\mathbb P$-name $\dot B$ of a cofinal subset of $\kappa$,
there exists a $q\le p$ that forces that the order-type of the intersection of $\nacc(C^q_{\gamma^q})$ and $\dot B$ is equal to $\gamma^q$.
\end{cor}
\begin{proof} Let $p$ and $\dot B$ as above be given. Define a sequence $\langle p_n\mid n<\omega\rangle$ by recursion on $n<\omega$, as follows.
First, using Lemma~\ref{lemma02}(1), find a $p_0\le p$ for which $\otp(C^{p_0}_{\gamma^{p_0}})<\gamma^{p_0}$.
Then, given $n<\omega$ such that $p_n$ has already been defined,
appeal to Lemma~\ref{lemma02}(2) to obtain a condition $p_{n+1}\le p_n$ such that all of the following hold:
\begin{enumerate}
\item $C^{p_n}_{\gamma^{p_n}}\sq C^{p_{n+1}}_{\gamma^{p_{n+1}}}$;
\item $\otp(C^{p_{n+1}}_{\gamma^{p_{n+1}}})<\gamma^{p_{n+1}}$;
\item $p_{n+1}$ forces that the intersection of $\nacc(C^{p_{n+1}}_{\gamma^{p_{n+1}}}\setminus \gamma^{p_{n}})$ and $\dot B$ has order-type $\ge\gamma^{p_n}$.
\end{enumerate}

Now, define a condition $q$ as follows:
\begin{itemize}
\item$\gamma^q:=\sup_{n<\omega}\gamma^{p_n}$;\footnote{Note that $\cf(\gamma^q)=\omega$. This is aligned with Remark~\ref{rmk33}.}
\item$C_\alpha^{q}:=C_\alpha^{p_n}$ for every $\alpha<\gamma^{q}$, using a large enough $n<\omega$;
\item$C^{q}_{\gamma^q}:=\bigcup\{ C^{p_n}_{\gamma^{p_n}}\mid n<\omega\}$;
\item$D^{q}:=\bigcup_{n<\omega}D^{p_{n}}\cup \{\gamma^q\}$.
\end{itemize}

As in the limit case from the proof of the previous Lemma, $q$ is a legitimate condition.
By Clauses (2) and (3), $q$ forces that the order-type of the intersection of $\nacc(C^q_{\gamma^q})$ and $\dot B$ is equal to $\gamma^q$.
\end{proof}

\begin{lemma}\label{1465} $\mathbb P$ forces that $\kappa^{<\kappa}=\kappa$.
\end{lemma}
\begin{proof} Work in $V[H]$ for $H$ a $\mathbb P$-generic over $V$.
By Lemma~\ref{lemma02}\ref{c1}, we may define a sequence $\vec C:=\langle C_\alpha\mid\alpha<\kappa\rangle$ by letting $C_\alpha:=C^p_\alpha$ for some nonempty condition $p\in H$ with $\gamma^p\ge\alpha$.
For all $\beta,\gamma<\kappa$, denote
$$X_{\beta,\gamma}:=\{\iota<\beta\mid \gamma+\omega\cdot\iota\in C_{\gamma+\omega\cdot\iota+\omega}\}.$$

\begin{claim} For every $X\in[\kappa]^{<\kappa}$,
there are $\beta,\gamma<\kappa$ such that $X=X_{\beta,\gamma}$.
\end{claim}
\begin{proof} By Lemma~\ref{lemma02}\ref{c3}, $V$ and $V[H]$ have the same bounded subsets of $\kappa$.
Thus, it suffices to prove that for every condition $p$ in $\mathbb P$ and every $X\in[\kappa]^{<\kappa}$ in $V$,
there are a condition $q\le p$ and ordinals $\beta,\gamma<\kappa$ such that $q$ forces that $X$ coincides with $X_{\beta,\gamma}$.
To this end, work in $V$ and let $p$ and $X$ be as above.
By possibly extending $p$, we may assume it is nonempty, say $p = (\langle C^p_\alpha \mid \alpha \le \gamma^p \rangle,D^p)$.
Find a large enough $\beta<\kappa$ such that $X\s\beta$.
Define $q = (\langle C^q_\alpha \mid \alpha \le \gamma^q \rangle,D^q)$
with $\gamma^{q}:=\gamma^{p}+\omega\cdot\beta$
and $D^{q}:=D^{p}\cup \{\gamma^{q}\}$, by letting for every $\alpha\le\gamma^q$:
$$C_\alpha^q:=
\begin{cases}
C_\alpha^p,&\text{if }\alpha\le\gamma^p;\\
\{\bar\alpha\},&\text{if }\alpha>\gamma^p\ \&\ \alpha=\bar\alpha+1;\\
\acc(\alpha\setminus\gamma^p),&\text{if }\alpha>\gamma^p\ \&\ \alpha\in\acc(\acc(\gamma^q+1));\\
\{\gamma^p+\omega\cdot\iota+n\mid n<\omega\},&\text{if }\alpha=\gamma^p+\omega\cdot\iota+\omega\ \&\ \iota\in X;\\
\{\gamma^p+\omega\cdot\iota+n\mid 0<n<\omega\},&\text{if }\alpha=\gamma^p+\omega\cdot\iota+\omega\ \&\ \iota\not\in X.
\end{cases}$$

It is clear that $q$ is a legitimate condition extending $p$ and forcing that $X_{\beta,\gamma^p}$ coincides with $X$.
\end{proof}
It immediately follows that $\kappa^{<\kappa}=\kappa$.
\end{proof}

\begin{cor}\label{1499} $\p_{<}(\kappa,2,{\sq},\kappa)$ holds in the generic extension by $\mathbb P$.
\end{cor}
\begin{proof} Work in $V[H]$ for $H$ a $\mathbb P$-generic over $V$.
By Lemma~\ref{lemma02}\ref{c1}, we may define a sequence $\vec C:=\langle C_\alpha\mid\alpha<\kappa\rangle$ by letting $C_\alpha:=C^p_\alpha$ for some nonempty condition $p\in H$ with $\gamma^p\ge\alpha$.
Also note that $D:=\bigcup\{ D^p\mid p\in H\}$ is a club in $\kappa$.
Clearly, $\vec C$ is a coherent $C$-sequence such that $R(\vec C)$ covers $D$.
For every $\alpha\in\acc(\kappa)$, let $X_\alpha:=\{ C_\alpha(\omega\cdot \iota+1)\mid \iota<\otp(\acc(C_\alpha))\}$.\footnote{Here, $C_\alpha(\xi)$
stands for the unique $\beta\in C_\alpha$ satisfying $\otp(C_\alpha\cap\beta)=\xi$.}

Let $\sigma\in\reg(\kappa)$. By Clause~(2) of Lemma~\ref{lemma02},
for every cofinal $B\s\kappa$, there exists an $\alpha\in E^\kappa_\sigma$ such that $X_\alpha$ is a cofinal subset of $\alpha$ and $\sup(X_\alpha\setminus B)<\alpha$.
Together with $\kappa^{<\kappa}=\kappa$ (that we obtain from Lemma~\ref{1465}), this easily implies that $\diamondsuit(E^\kappa_\sigma)$ holds.\footnote{See the proof of \cite[Claim~6.6.2]{paper58}.}
In particular, we may let $\Phi^\omega:\mathcal K(\kappa)\rightarrow\mathcal K(\kappa)$ be the postprocessing function given by \cite[Lemma~3.8]{paper28}.
Define a $C$-sequence $\vec{C^\bullet}:=\langle C^\bullet_\alpha\mid\alpha<\kappa\rangle$ by letting $C_0^\bullet:=\emptyset$,
$C_{\alpha+1}^\bullet:=\{\alpha\}$ for every $\alpha<\kappa$ and $C_\alpha^\bullet:=\Phi^\omega(C_\alpha)$ for every $\alpha\in\acc(\kappa)$.
Then $\vec{C^\bullet}$ is coherent $C$-sequence such that $R(\vec{C^\bullet})\supseteq R(\vec C)\supseteq D$.
Finally, given a sequence $\langle B_i\mid i<\kappa\rangle$ of cofinal subsets of $\kappa$,
let $G$ be the corresponding stationary set given by \cite[Lemma~3.8]{paper28}.
By Corollary~\ref{cor14}, we may now fix arbitrarily large $\gamma\in\acc(\kappa)$ such that $\otp(\nacc(C_\gamma)\cap G)=\gamma$.
It then follows from Clause~(2) of \cite[Lemma~3.8]{paper28} that for any such $\gamma$, for every $i<\gamma$, $\sup(\nacc(C_\gamma^\bullet)\cap B_i)=\gamma$.
So, $\vec{C^\bullet}$ witnesses $\p^-_{<}(\kappa,2,{\sq},\kappa)$.
\end{proof}

\begin{q} Suppose that $\VisL$ and $\kappa$ is an inaccessible that is non-Mahlo.
Does $\p_{<}(\kappa,2,{\sq},\kappa)$ hold?
\end{q}

An affirmative answer may possibly follow from the arguments of \cite[\S4]{paper29}.

\section{Clause~(2) of Theorem~\ref{thmb}}\label{section:thmb2}

Throughout this section, $\kappa$ stands for a regular uncountable cardinal.

All necessary background on set-theoretic trees may be found in standard textbooks, as well as at \cite[\S2]{paper22}.
In addition, for our purpose, we shall need the following definition.
\begin{defn}[Todor\v{c}evi\'c, {\cite[p.~266]{MR908147}}]\label{defn41} A tree $(T,{<_T})$ of height $\kappa$ is \emph{special}
iff there exists a map $g:T\rightarrow T$ satisfying the following:
\begin{itemize}
\item for every non-minimal $x\in T$, $g(x)<_T x$;
\item for every $y\in T$, $g^{-1}\{y\}$ is covered by less than $\kappa$ many antichains.
\end{itemize}
\end{defn}
\begin{remark}\label{rmk42} In case that $\kappa=\lambda^+$ is a successor cardinal, a tree of height $\kappa$ is special (in the above sense)
iff it may be covered by $\lambda$ many antichains \cite[Theorem~14]{MR793235}.
\end{remark}
\begin{defn}
Let $\mathbb{Q}_\kappa$ denote the linear order consisting of all nonempty finite sequences of ordinals in $\kappa$,
with the ordering $q\ll p$ iff either $p\subsetneq q$ or $q(n)<p(n)$ for the least $n<\omega$ such that $q(n)\neq p(n)$.
\end{defn}

The next easy fact is an immediate corollary to a claim from \cite[\S2]{paper71}, but we give the details here for completeness.

\begin{prop}\label{p33} Suppose that $\mathbf T=(T,{<_T})$ is a tree of height $\kappa$
and there exists a map $f:T\rightarrow\mathbb Q_\kappa$ satisfying the following two requirements:
\begin{enumerate}[label=\textup{(\arabic*)}]
\item $f$ is strictly increasing, i.e., for all $x<_Tx'$ in $T$, $f(x)\ll f(x')$;
\item\label{332} the set $\{ \alpha<\kappa \mid f[T_\alpha]\s {}^{<\omega}\alpha \}$ covers a club.
\end{enumerate}

Then $\mathbf T$ is special.
\end{prop}
\begin{proof} Fix a club $C\s\kappa$ such that $f[T_\alpha]\s {}^{<\omega}\alpha$ for every $\alpha\in C$.
Fix a bijection $\pi:{}^{<\omega}\kappa\leftrightarrow\kappa$,
and consider the club $D:=\{\alpha\in\acc(\kappa)\cap C\mid \pi[{}^{<\omega}\alpha]=\alpha\}$.
Clearly, $\pi(f(x))<\h(x)$ for every $x\in T$ with $\h(x)\in D$.

For all $\alpha<\kappa$ and $x\in T$ with $\h(x)>\alpha$,
let $x\restriction\alpha$ denote the unique $y<_Tx$ with $\h(y)=\alpha$.
Define a map $g:T\rightarrow T$, as follows:
$$g(x):=\begin{cases}
x\restriction (\pi(f(x))+1),&\text{if }\h(x)\in D;\\
x\restriction \sup(D\cap\h(x)),&\text{otherwise}.
\end{cases}$$

It is clear that $g(x)=x$ for every minimal $x\in T$ and that $g(x)<_T x$ for every non-minimal $x\in T$.
Now, given $y\in T$, there are two cases to consider:

$\br$ If $\h(y)$ is a successor ordinal,
then for the unique $p\in{}^{<\omega}\kappa$ such that $\h(y)=\pi(p)+1$, for all $x,x'\in g^{-1}\{y\}$,
we have $f(x)=p=f(x')$ and hence $x,x'$ are incomparable. That is, $(g^{-1}\{y\},{<_T})$ is an antichain.

$\br$ Otherwise, $\epsilon:=\h(y)$ is a (possibly zero) limit ordinal, and then $g^{-1}\{y\}$ is a subset of $\{ x\in T\mid \h(x)<\min(D\setminus(\epsilon+1))\}$,
which is a set consisting of fewer than $\kappa$ many levels of $T$.
In particular, $g^{-1}\{y\}$ is the union of less than $\kappa$ many antichains of $(T,{<_T})$.
\end{proof}

\begin{defn}[{\cite[Definition~2.3]{paper23}}] A set $T$ is a \emph{streamlined tree} iff there exists some cardinal $\theta$ such that $T\s{}^{<\theta}H_\theta$
and, for all $t\in T$ and $\beta<\dom(t)$,
$t\restriction\beta\in T$.
\end{defn}

We identify a streamlined tree $T$ with the abstract set-theoretic tree $(T,{\subsetneq})$.

\begin{defn}\label{def45}
For a subset $T\s{}^{<\kappa}H_\kappa$ and a positive integer $n$,
denote $T^n:=\{\vec x:n\stackrel{1-1}{\longrightarrow} T\mid i\mapsto\dom(\vec{x}(i))\text{ is constant}\}$.
The ordering $<_{T^n}$ of $T^n$ is defined as follows:
$$\vec x<_{T^n}\vec y\iff \bigwedge_{i<n}\vec x(i)\subsetneq\vec y(i).$$
\end{defn}

Note that the sequences in $T^n$ are injective.
This requirement was not imposed back in Section~\ref{secthma} (see Definition~\ref{square-def}) because it is inessential in the context in which the main notion of a square is $(X_T)^2$,
and in that section we opted to prefer simplicity over generality.

\begin{defn}[Derived trees] For a streamlined tree $T$ and a positive integer $n$,
an \emph{$n$-derived tree of $T$} is a collection of the form
$$T(w_0,\ldots,w_{n-1}):=\{ (x_0,\ldots,x_{n-1})\in T^n\mid \forall i<n\,(x_i\cup w_i\in T)\},$$
for some node $(w_0,\ldots,w_{n-1})\in T^n$.
\end{defn}
\begin{defn} A streamlined $\kappa$-tree $T$ is \emph{$(n+1)$-free} iff all of its $n$-derived trees are $\kappa$-Souslin.
\end{defn}
\begin{remark} $(T^n,{<_{T^n}})$ is the union of all $n$-derived trees of $T$.
In particular, if $(T^n,{<_{T^n}})$ admits a map as in Proposition~\ref{p33}, then all $n$-derived trees of $T$ are special.
\end{remark}

A forcing construction of an $n$-free $\aleph_1$-Souslin tree all of whose derived trees of dimension $n$ are special was given in \cite[Corollary~5.5]{MR4080091}.
A uniform combinatorial construction of an $n$-free $\lambda^+$-Souslin tree $T$ such that $(T^n,{<_{T^n}})$ admits a strictly increasing map to $\mathbb Q_\lambda$ was given in \cite[\S6]{paper65},
assuming $\p_\lambda(\lambda^+,2,{\sq},\lambda^+)$.
The construction given there is abstract enough to generalize to limit cardinals $\kappa$ assuming $\p_\kappa(\kappa,2,{\sq},\kappa)$, where the map now goes to $\mathbb Q_\kappa$.
In view of Proposition~\ref{p33}, the challenge remaining is to ensure that Clause~\ref{332} of the proposition will hold for this map.
This naturally depends on the very $C$-sequence used in the construction,
and this is where Clause~\ref{321} of Definition~\ref{def32} demonstrates its utility.

\begin{thm}\label{thm710b}
Suppose that $\p_{<}(\kappa,2,\allowbreak{\sq},\kappa)$ holds, and $\chi\in[2,\omega)$.
Then there exists a $\chi$-free, streamlined $\kappa$-Souslin tree $T$ and strictly increasing $f:T^\chi\rightarrow\mathbb Q_\kappa$ such that
$f[(T_\alpha)^\chi]\s {}^{<\omega}\alpha$ for club many $\alpha<\kappa$.
In addition, $T$ is slim, prolific and club-regressive.\footnote{The definitions of all additional properties may found at \cite[\S4.2]{paper65}.}
\end{thm}
\begin{proof} Let $\vec C=\langle C_\alpha\mid\alpha<\kappa\rangle$ be a $C$-sequence satisfying Clauses \ref{321}--\ref{323} of Definition~\ref{def32}.
Derive $\vec D=\langle D_\alpha\mid\alpha\in\acc(\kappa)\rangle$ from $\vec C$ as in the beginning of the proof of \cite[Theorem~6.11]{paper65}
and note that $\vec D$ is a coherent $C$-sequence with $R(\vec D)=R(\vec C)$.
Running the exact same abstract proof of \cite[Theorem~6.11]{paper65}, modulo replacing $\lambda$ by $\kappa$ throughout,
we obtain a $\chi$-free, slim, prolific, club-regressive, streamlined
$\kappa$-Souslin tree $T$ and a strictly increasing map $f:T^\chi\rightarrow\mathbb Q_\kappa$.
Thus, we are left with verifying that $f[(T_\alpha)^\chi]\s{}^{<\omega}\alpha$ for club many $\alpha<\kappa$.

As $\chi$ is finite and $T$ is a $\kappa$-tree, we may fix a club $E\s\kappa$ such that,
for every $\alpha \in E$ and every $\gamma<\alpha$, $f[(T_\gamma)^\chi]\s{}^{<\omega}\alpha$. As $R(\vec D)$ covers a club, we may shrink $E$ and also assume it is a subset of $R(\vec D)$
and that all ordinals in $E$ are indecomposable, that is, $\beta+\gamma<\alpha$ for all $\alpha \in E$ and all $\beta,\gamma<\alpha$.

Now, given $\alpha\in E$ and $\vec{w}=\langle w_i\mid i<\chi\rangle$ in $(T_\alpha)^\chi$,
we remind the reader that the definition of $f_\alpha(\vec w)$ (right after \cite[Claim~6.11.6]{paper65}) went as follows.
First, for each $i<\chi$, we found an $x_i\in T\restriction D_\alpha$ of minimal height such that $\mathbf{b}_{x_i}^\alpha=w_i$.
Then we set $\gamma:=\sup\{\dom(x_i)\mid i<\chi\}$,
and then we denoted $(\varphi_1\circ f_\gamma)(\langle w_i\restriction\gamma\mid i<\chi\rangle)$ by $p^\smallfrown\langle\xi\rangle$.
Finally, as $\chi$ is an integer, it follows that $\gamma<\alpha$. Thus, there are two cases to consider:
\begin{itemize}[label=$\blacktriangleright$]
\item If $\alpha\in A(\vec{D})$, then $f_\alpha(\vec w)=p$ which is an element of ${}^{<\omega}\alpha$, since $\gamma<\alpha$ and $\alpha\in E$.
\item Otherwise, $f_\alpha(\vec w)=p^\smallfrown\langle\xi+\sigma\rangle$, where $\sigma := \otp(D_\alpha\setminus(\gamma+1))$.
As $\alpha\in R(\vec C)\setminus A(\vec C)$ and since $\vec D$ is coherent, we infer that $\otp(D_\alpha)<\alpha$.
In addition, since $\alpha\in E$ and $\gamma<\alpha$, $p^\smallfrown\langle\xi\rangle$ is in ${}^{<\omega}\alpha$.
Altogether, $p^\smallfrown\langle\xi+\sigma\rangle$ is an element of ${}^{<\omega}\alpha$. \qedhere
\end{itemize}
\end{proof}

\section*{Acknowledgments}
The first author was supported by the Shamoon College of Engineering Young Research Grant YR/08/Y21/T2/D3.
The second and third authors were supported by the Israel Science Foundation (grant agreement 203/22).

Portions of this paper were presented at the BIU Set Theory seminar (by the third author, April 2024,
and by the second author, January 2025), at the Toronto Set Theory seminar (by the first author, October 2025),
and at the HUJI Set Theory seminar (by the second author, November--December 2025), and at the RIMS Workshop on Set Theory (by the first author, December 2025).
We thank the participants for their feedback.

\end{document}